\documentclass[11pt]{scrartcl}
\usepackage[utf8x]{inputenc}
\usepackage[T1]{fontenc}
\usepackage{lmodern}
\usepackage[english]{babel}

\usepackage[authoryear]{natbib}

\usepackage{amsmath,amssymb,amsthm,graphicx}

\newtheorem{lemma}{Lemma}[section]
\theoremstyle{remark}

\usepackage{todonotes}
\usepackage{hyperref}
\usepackage{mathtools}

\usepackage{scalerel}
\usepackage{tikz}
\usetikzlibrary{svg.path}
\usetikzlibrary{positioning}
\usepackage{relsize}
\usepackage[linesnumbered]{algorithm2e}
\usepackage{booktabs}
\usepackage[justification=centering]{caption}
\usepackage{subfig}

\DeclareMathOperator{\diag}{diag}
\DeclareMathOperator{\rank}{rank}

\DeclarePairedDelimiter\floor{\lfloor}{\rfloor}
\DeclarePairedDelimiter\set\{\}
\newcommand{\R}{\mathbb{R}}

\DeclareMathOperator{\conv}{conv}
\DeclareMathOperator{\st}{s.t.}

\usepackage{enumitem}

\newcommand{\CWpiv}[2]{CW_{#1}(#2)}
\newcommand{\CWpiG}[1]{CW_{#1}(G)}
\newcommand{\CWG}{CW(G)}
\newcommand{\Pn}{\Pi_n}
\newcommand{\LOn}{LO(n)}
\newcommand{\LOTwon}{LO^{2}(n)}
\newcommand{\LOTwo}[1]{LO^{2}(#1)}
\newcommand{\STAB}{STAB}

\definecolor{orcidlogocol}{HTML}{A6CE39}
\tikzset{
	orcidlogo/.pic={
		\fill[orcidlogocol] svg{M256,128c0,70.7-57.3,128-128,128C57.3,256,0,198.7,0,128C0,57.3,57.3,0,128,0C198.7,0,256,57.3,256,128z};
		\fill[white] svg{M86.3,186.2H70.9V79.1h15.4v48.4V186.2z}
		svg{M108.9,79.1h41.6c39.6,0,57,28.3,57,53.6c0,27.5-21.5,53.6-56.8,53.6h-41.8V79.1z M124.3,172.4h24.5c34.9,0,42.9-26.5,42.9-39.7c0-21.5-13.7-39.7-43.7-39.7h-23.7V172.4z}
		svg{M88.7,56.8c0,5.5-4.5,10.1-10.1,10.1c-5.6,0-10.1-4.6-10.1-10.1c0-5.6,4.5-10.1,10.1-10.1C84.2,46.7,88.7,51.3,88.7,56.8z};
	}
}

\newbox{\myorcidaffilbox}
\sbox{\myorcidaffilbox}{\large\mbox{\scalerel*{
			\begin{tikzpicture}[yscale=-1,transform shape]
			\pic{orcidlogo};
			\end{tikzpicture}
		}{|}}}
\newcommand\orcidicon[1]{\href{https://orcid.org/#1}{\usebox{\myorcidaffilbox}}}

\providecommand{\keywords}[1]
{
	\small	
	\textbf{Keywords:} #1
}
\usepackage{authblk} 

\newcommand{\revision}[1]{{\color{black}#1}}

\usepackage[normalem]{ulem}

\title{Strong SDP based bounds on the cutwidth of a graph%
	\thanks{This research was supported by the Austrian Science
          Fund (FWF): DOC~78 and by the Johannes Kepler University
          Linz,
	Linz Institute of Technology (LIT): LIT-2021-10-YOU-216.}
}

\author[1]{Elisabeth Gaar \orcidicon{0000-0002-1643-6066}
}
\author[2]{Diane Puges \orcidicon{0000-0002-3864-0762}
}
\author[3]{Angelika Wiegele \orcidicon{0000-0003-1670-7951}
	}
\affil[1]{Johannes Kepler 
	University Linz, Altenberger Stra{\ss}e~69, \newline
	4040 Linz, Austria,
	\href{mailto:elisabeth.gaar@jku.at}{elisabeth.gaar@jku.at}}
\affil[2]{Alpen-Adria-Universit\"at Klagenfurt,
	Universit\"atsstra{\ss}e~65--67, 9020 Klagenfurt, Austria, 
	\href{mailto:diane.puges@aau.at}{diane.puges@aau.at}}
\affil[3]{Alpen-Adria-Universit\"at Klagenfurt,
        Universit\"atsstra{\ss}e~65--67, 9020 Klagenfurt, Austria \& 
        Universit\"at zu K\"oln,
        Albertus-Magnus-Platz, 50923 K\"oln, Germany, 
	\href{mailto:angelika.wiegele@aau.at}{angelika.wiegele@aau.at}}

\date{}

\begin{document}

\maketitle

\begin{abstract}
Given a linear ordering of the vertices of a graph, the cutwidth 
of a vertex~$v$ with respect to this ordering is the number of edges from any 
vertex before~$v$ (including $v$) to any vertex after~$v$ in this ordering. 
The cutwidth of an ordering is the maximum cutwidth of any vertex with respect 
to this ordering.
We are interested in finding the cutwidth of a graph, that is, 
the minimum cutwidth over all orderings, 
which is an NP-hard problem. 
In order to approximate the cutwidth of a given
graph, we present a semidefinite relaxation. 
We identify several classes of valid inequalities and equalities that we
use to strengthen the semidefinite relaxation. These classes are
on the one hand the well-known 3-dicycle equations and the triangle
inequalities and on the other hand we obtain inequalities from the
squared linear ordering polytope and via lifting the linear
ordering polytope.
The solution of the semidefinite program serves to obtain a lower
bound and also to construct a feasible solution and thereby having an
upper bound on the cutwidth. 

In order to evaluate the quality of our bounds, we perform numerical
experiments on graphs of different sizes and densities. It turns out
that we produce high quality bounds for graphs of medium size
independent of their density in reasonable time. Compared to that, 
obtaining bounds for dense instances of the same quality is out of
reach for solvers using integer linear programming techniques.

\keywords{Cutwidth, linear ordering, semidefinite programming, 
combinatorial optimization}
\end{abstract}

\section{Introduction}
\label{sec:intro}

Several graph parameters are determined by finding an 
arrangement of the vertices of a graph on a straight line 
having a certain objective in mind. 
Depending on the objective, these parameters are, for instance, the
treewidth, the pathwidth, the bandwidth or the cutwidth of a graph.
Computing these graph parameters is necessary for several
applications, e.g., when a 
certain layout has to be found (VLSI design\revision{, see} 
\cite{raspaud1995debruijn}), 
in automatic graph
drawing\revision{, see} \cite{diaz2002survey}, or in many versions of network 
problems arising in energy or
logistics. 
All these applications ask for algorithms to practically compute these
parameters. However, this leads to NP-hard
optimization problems and therefore algorithms for approximating these
parameters in terms of lower and upper bounds are required.
The parameter that is of our interest in this work is the cutwidth of a graph.

\paragraph{Definitions.}

The minimum cutwidth problem (MCP) can be defined 
as follows.
We consider an undirected graph $G=(V,E)$ with vertex set $V$ and 
edge set $E$ and assume without loss of generality that the set 
of vertices $V$ is $V = \set{1, \dots, n}$. 
Furthermore, the set of all permutations of $ \set{1, \dots, n}$ is denoted by 
$\Pn$. 
In 
any permutation $\pi \in \Pn$ of the vertices of $G$, the position of a vertex 
$v \in V$ in $\pi$ 
is given 
by $\pi(v)$. 
Note that any linear ordering of the vertices $V$ of $G$ can be represented by 
a permutation $\pi \in \Pn$. 

The cutwidth $\CWpiv{\pi}{v}$ of a vertex $v$ with respect to the permutation $\pi$ is the number 
of edges 
$\set{u,w} \in E$ such that $\pi(u) \leq \pi(v) < \pi(w)$ holds, i.e., the 
cutwidth 
of $v$ in $\pi$ 
is the number of edges from any vertex before $v$ (including $v$) to any vertex 
after $v$ (not including $v$) in a linear ordering of the vertices according to 
the permutation $\pi$.
Furthermore, the cutwidth $\CWpiG{\pi}$ of a graph $G$ with respect to $\pi$ is 
the 
maximum cutwidth of any vertex with respect to the permutation $\pi$, so 
\begin{align*}
\CWpiG{\pi} = \max_{v \in V} \CWpiv{\pi}{v}
\end{align*} 
holds. Finally, the cutwidth $\CWG$ of a graph $G$ is the minimum 
cutwidth of $G$ with respect to $\pi$ over all possible permutations $\pi$, 
i.e., 
\begin{align*}
\CWG = \min_{\pi \in \Pn} \CWpiG{\pi}. 
\end{align*}

An obvious lower bound on the cutwidth is given by $$\CWG \geq \floor{\frac{\Delta(G)+1}{2}},$$ 
where $\Delta(G)$ is the maximum degree of any vertex in $V$.
Indeed, let us denote by $v_{\max}$ a vertex of G with degree
$\Delta(G)$. Then, for every linear ordering $\pi$, every vertex in the
neighborhood of $v_{\max}$ is counted either in $\CWpiv{\pi}{
  v_{\max}}$ or in  $\CWpiv{\pi}{ u_{\max}}$, where $u_{\max}$ is the
vertex directly preceding $v_{\max}$ in $\pi$. It follows that either
$\CWpiv{\pi}{ v_{\max}}$ or $\CWpiv{\pi}{ u_{\max}}$ has to be
greater or equal to ${\frac{\Delta(G)}{2}}$, and due to the integrality
of $\CWG$ we obtain the lower bound $\floor{\frac{\Delta(G)+1}{2}}$. 

Among connected graphs with $n$ vertices, the graphs with the smallest
cutwidth are the paths, which have a cutwidth of~1. The graphs with the
largest cutwidth are the complete graphs $K_n$, with $CW(K_n) =
\floor{\frac{n²}{4}}$.

\paragraph{Related literature.}
The MCP has been investigated in several aspects \revision{and in
  several contexts. 
Inspired from results in the topology of manifolds,
  \cite{kloeckner2009} explores lower bounds depending on the sparsity
  and the degeneracy of the underlying graph.}
Next to theoretical properties of the cutwidth,  
connections between the cutwidth of a graph $G$ with its treewidth
$TW(G)$ and its pathwidth $PW(G)$ have been explored by several
authors. It is known that $\CWG \leq \Delta(G) PW(G)$ \revision{from}
\cite{chung1989graphs}, $\CWG \geq TW(G)$ \revision{from} 
\cite{bodlaender1986classes}, and
$\CWG = \mathcal{O}((\log n)\Delta(G) TW(G))$ \revision{from} 
\cite{korach1993tree}.
It follows that if $TW(G)$ and $ \Delta(G)$ are bounded by constants, 
$\CWG = \mathcal{O}(\log n)$. Furthermore, if $(X, T)$ is a tree decomposition
of $G$ with treewidth $k$, then $\CWG \leq (k+1)\Delta(G) CW(T)$. 

As for computing the cutwidth, there exist polynomial time algorithms for
certain graph classes, see
e.g.~\cite{heggernes2011splitgraphs,heggernes2012cutwidth,yannakakis1985polynomial}.
\cite{giannopoulou2019cutwidth} design a
fixed-parameter algorithm for computing the cutwidth that runs in time
$2\mathcal{O}(k^2\log{} k)n$, where $k$ is the cutwidth.
\revision{In a more general setting, \cite{bodlaender2012vertexorder}
  discuss exponential time algorithms for vertex ordering problems,
  including the MCP. A relation of computing the cutwidth and the
  pathwidth via the so-called locality number, a structural parameter
  for strings, has been investigated by~\cite{casel2019connections}.}

Another way to solve the MCP is to model the optimization
problem as a mixed-integer linear
program (MILP). This has been considered by 
\cite{luttamaguzi2005} and by 
\cite{LopezLoces2014}. Moreover, in the PhD
thesis of 
\cite{coudert2016} MILP formulations for linear
ordering problems, among them the cutwidth, pathwidth and bandwidth
problem, are given. Therein, different formulations for these problems
are introduced and compared to each other. 
However, all these algorithms can only deal with sparse instances of
small size. Indeed, \revision{by} \cite{coudert2016} results are given for 
(very) sparse
graphs only with  $|V| \in \{16, \dots, 24\}$.

\revision{\cite{Marti2013bb} introduce a branch-and-bound algorithm
  using lower bounds on the cutwidth of partial solutions and a
  greedy randomized adaptive search procedure (GRASP) to compute upper
  bounds. Applying the metaheuristic adaptive large neighborhood
  search for obtaining orderings with a small cutwidth has been
  introduced by \cite{santos2021}.

} 

A further line of research is to apply semidefinite programming (SDP) to
optimization problems that deal with orderings of the vertices, \revision{see}
e.g.~\cite{buchheim2010, hungerlaender2013ordering}.
In particular, SDP based methods proved to be very successful when
applied to the \revision{single row} facility layout problem, \revision{see} \cite{fischer2019layout, 
schwiddessen22}, \revision{which falls into this category as its goal is to 
order facilities on a straight line in the best way according to some objective function}.
SDP based bounds for the bandwidth problem are introduced 
in~\cite{rendl2021lower}. 
To the best of our knowledge there have been no attempts so far in using
semidefinite programming to tackle the MCP.

\paragraph{Contribution and outline.}
In this paper we present a novel relaxation of the MCP that uses semidefinite
programming. We introduce several valid inequalities that we include
in the SDP in a cutting-plane fashion to strengthen the lower bound on the 
cutwidth.
Moreover, we introduce a heuristic that uses the optimizer of
the SDP to obtain feasible solutions and thereby providing an upper
bound on the cutwidth.
Our computational experiments confirm that we can obtain tight
bounds in reasonable time and that the run time of our algorithms are not
sensitive concerning the density of the graph.

The rest of this paper is structured as follows.
In Section~\ref{sec:sdpsection} we introduce a basic SDP relaxation for the
MCP and provide several linear constraints that can be
used to strengthen the basic SDP relaxation.
We describe in detail the algorithm for computing the lower bounds
arising from the SDP as well as \revision{a} heuristic to obtain upper bounds in Section~\ref{sec:algo}.
Computational results of our new algorithms are given in
Section~\ref{sec:computations}, and Section~\ref{sec:conclusion} concludes.

\section{New bounds for the minimum cutwidth problem}\label{sec:sdpsection}
The aim of this section is to introduce a new SDP relaxation for the MCP and to 
present ways to strengthen this relaxation.

\subsection{Our new basic SDP relaxation}
\label{sec:sdp_formulation}
We follow the approach of 
\cite{buchheim2010} for 
the quadratic linear ordering problem in order to derive an SDP formulation of 
the MCP. This is a promising endeavor, as the feasible region of both the 
quadratic linear ordering problem and the MCP consist of permutations.

Towards this end, let $A=(a_{ij})_{1 \leq i, j \leq n}$ be the adjacency matrix 
of the graph $G$, 
i.e.,
$a_{ij}=1$ if and only if the edge $\set{i,j}$ is in the set of edges $E$ of 
the graph $G$.  
To  represent a permutation $\pi$, 
\cite{buchheim2010} introduce a binary variable 
$\chi^\pi_{ij}$ for any $ 1 \leq i < j \leq n$ that indicates whether $\pi(i) < 
\pi(j)$ holds, so in total they have a vector $\chi^\pi$ consisting of  
$\binom{n}{2}$ binary variables. 
They define the linear ordering polytope $\LOn$ as the convex hull of all 
vectors $\chi^\pi$, formally
\begin{align*}
\LOn = \conv \set{ \chi^\pi \colon \pi \in \Pn }, 
\end{align*}
which is a subset of $\R^{\binom{n}{2}}$.
As a consequence, the elements of the set $\LOn \cap \set{0,1}^{\binom{n}{2}}$ 
are exactly 
all 
vectors representing permutations of $\set{1, \dots, n}$.

Let $x = (x_{ij})_{1 \leq i < j \leq n}$ be a vector in $\LOn \cap 
\set{0,1}^{\binom{n}{2}}$ 
representing a 
permutation $\pi \in \Pn$. 
By definition, the binary variable $x_{ij}$ is equal to $1$ if and only if $i$  
is before $j$ in $\pi$.
Then, it can be checked that for any vertex $v$ of $G$,
\begin{align}
\label{eq:cutwtithdAsSum}
\CWpiv{\pi}{v}  = &
\sum_{\substack{u \in V\\u < v}}
\sum_{\substack{w \in V\\w > v}} 
 a_{uw}x_{uv}x_{vw} + 
 \sum_{\substack{u \in V\\u > v}}
 \sum_{\substack{w \in V\\w > v\\w \neq u}} 
 a_{uw}(1-x_{vu})x_{vw}  + 
 \sum_{\substack{u \in V\\u < v}}
 \sum_{\substack{w \in V\\w < v\\w \neq u}} 
 a_{uw}x_{uv}(1-x_{wv}) \nonumber \\ 
 & + 
 \sum_{\substack{u \in V\\u > v}}
 \sum_{\substack{w \in V\\w < v}} 
 a_{uw}(1-x_{vu})(1-x_{wv}) 
  + 
 \sum_{\substack{w \in V\\w > v}}
 a_{vw}x_{vw} + 
  \sum_{\substack{w \in V\\w < v}}
 a_{vw}(1-x_{wv})
\end{align}
holds.
The first four terms of this expression count the number of edges from 
any 
 vertex before $v$ to any vertex after $v$ in the permutation, both not 
 including $v$. These four terms are necessary, 
 as only one of the variables 
 $x_{uv}$ (if $u<v$) and $x_{vu}$ (if $u>v$), and one of the variables $x_{wv}$ 
 (if $w<v$) and $x_{vw}$ (if $w>v$) exist.
The last two terms of~\eqref{eq:cutwtithdAsSum} count the edges from $v$ to 
any 
vertex after $v$. 

By expanding~\eqref{eq:cutwtithdAsSum}, grouping the constant, linear and 
quadratic terms together, and by combining three times two sums as one, we can 
rewrite~\eqref{eq:cutwtithdAsSum} as
\begin{align}
\label{eq:cutwtithdAsSumExpanded}
\CWpiv{\pi}{v}  &= 
\sum_{\substack{w \in V\\w < v}}
\sum_{\substack{u \in V\\u \geq v}} 
a_{uw} 
+
\sum_{\substack{w \in V\\w > v}}
\sum_{\substack{u \in V\\u \geq v\\{u \neq w}}} 
a_{uw}x_{vw}
+\sum_{\substack{u \in V\\u < v}}
 \sum_{\substack{w \in V\\w < v\\{w \neq u}}} 
 a_{uw}x_{uv}\\
&\phantom{\geq}- \sum_{\substack{u \in V\\u > v}}
\sum_{\substack{w \in V\\w < v}} 
a_{uw}(x_{vu} + x_{wv}) 
- \sum_{\substack{w \in V\\w < v}}
a_{vw}x_{wv} \nonumber
+ 
2\sum_{\substack{u \in V\\u < v}}
\sum_{\substack{w \in V\\w > v}} 
a_{uw}x_{uv}x_{vw} \\
&\phantom{\geq}
-
\sum_{\substack{u \in V\\u > v}}
\sum_{\substack{w \in V\\w > v\\{w \neq u}}} 
a_{uw}x_{vu}x_{vw} 
-
\sum_{\substack{u \in V\\u < v}}
\sum_{\substack{w \in V\\w < v\\{w \neq u}}} 
a_{uw}x_{uv}x_{wv},\nonumber
\end{align}
and as a result, the MCP can be written as
\begin{subequations}
\begin{alignat}{3}
& \min &~\alpha   \\ 
& \st~ &  \alpha &\geq  
\sum_{\substack{w \in V\\w < v}}
\sum_{\substack{u \in V\\u \geq v}} 
a_{uw} 
+
\sum_{\substack{w \in V\\w > v}}
\sum_{\substack{u \in V\\u \geq v\\{u \neq w}}} 
a_{uw}x_{vw}
+\sum_{\substack{u \in V\\u < v}}
\sum_{\substack{w \in V\\w < v\\{w \neq u}}} 
a_{uw}x_{uv}\\
&&&\phantom{\geq}- \sum_{\substack{u \in V\\u > v}}
\sum_{\substack{w \in V\\w < v}} 
a_{uw}(x_{vu} + x_{wv}) 
- \sum_{\substack{w \in V\\w < v}}
a_{vw}x_{wv} \nonumber
+ 
2\sum_{\substack{u \in V\\u < v}}
\sum_{\substack{w \in V\\w > v}} 
a_{uw}x_{uv}x_{vw} \quad \\
&&&\phantom{\geq}
-
\sum_{\substack{u \in V\\u > v}}
\sum_{\substack{w \in V\\w > v\\{w \neq u}}} 
a_{uw}x_{vu}x_{vw} 
-
\sum_{\substack{u \in V\\u < v}}
\sum_{\substack{w \in V\\w < v\\{w \neq u}}} 
a_{uw}x_{uv}x_{wv}
&& \forall v \in V \nonumber \\       
&& x &\in  \LOn \cap 
\set{0,1}^{\binom{n}{2}}.
\end{alignat}
\end{subequations}
This is an integer program with a linear objective function and quadratic 
constraints, which is perfectly suited for deriving an SDP 
relaxation. To do so, we introduce the matrix variable $X = 
(X_{ij,k\ell})_{\substack{1 \leq i < j \leq n \\1 \leq k < \ell \leq n}}$. In 
particular, $X_{ij,k\ell}$ represents the product $x_{ij}x_{k\ell}$. Then the 
following SDP is a relaxation of the MCP.
\begin{subequations}
	\label{eq:sdpCutwidth}\label{SDP_init}
	\begin{alignat}{3}
	& \min & \alpha  \\ 
& \st~ &  \alpha & \geq  \label{eq:cutwidth}
\sum_{\substack{w \in V\\w < v}}
\sum_{\substack{u \in V\\u \geq v}} 
a_{uw} 
+
\sum_{\substack{w \in V\\w > v}}
\sum_{\substack{u \in V\\u \geq v\\{u \neq w}}} 
a_{uw}x_{vw}
+\sum_{\substack{u \in V\\u < v}}
\sum_{\substack{w \in V\\w < v\\{w \neq u}}} 
a_{uw}x_{uv}\\
&&& \phantom{\geq} - \sum_{\substack{u \in V\\u > v}}
\sum_{\substack{w \in V\\w < v}} 
a_{uw}(x_{vu} + x_{wv}) 
- \sum_{\substack{w \in V\\w < v}}
a_{vw}x_{wv} \nonumber
+ 
2\sum_{\substack{u \in V\\u < v}}
\sum_{\substack{w \in V\\w > v}} 
a_{uw}X_{uv,vw} \quad \\
&&& \phantom{\geq}
-
\sum_{\substack{u \in V\\u > v}}
\sum_{\substack{w \in V\\w > v\\{w \neq u}}} 
a_{uw}X_{vu,vw} 
-
\sum_{\substack{u \in V\\u < v}}
\sum_{\substack{w \in V\\w < v\\{w \neq u}}} 
a_{uw}X_{uv,wv}
\qquad && \forall v \in V \nonumber \\       	
	&& x &=  \diag(X) \label{eq:diag}\\
	&& \bar{X} &= \begin{pmatrix} 1 & x^\top \\ x & X\end{pmatrix} 
	\label{eq:psd1}  \\
	&& \bar{X} &\succeq 0.&& \label{eq:psd2} 
	\end{alignat}
\end{subequations}

Note that $x = (x_{ij})_{1 \leq i < j \leq n}$ is a vector of dimension 
$\binom{n}{2}$ and $\bar{X}$ is a square matrix with $(\binom{n}{2}+1)$ columns 
and rows. 
Thus, the SDP~\eqref{eq:sdpCutwidth} has a matrix variable of dimension  
$(\binom{n}{2}+1)$
and $n + \binom{n}{2} + 1$ constraints.
The $n$ constraints~\eqref{eq:cutwidth} make sure that $\alpha$ is at least the 
cutwidth of each vertex. 
The $\binom{n}{2}+1$ constraints~\eqref{eq:diag} together with the constraints~\eqref{eq:psd1} 
and~\eqref{eq:psd2} represent the  
relaxation 
of $X - xx^\top = 0$
to $X - xx^\top \succeq 0$ and taking the Schur complement. 
Due to the fact 
that $x$ is binary, $X - xx^\top = 0$ also implies that~\eqref{eq:diag} has to 
hold.

\subsection{Strengthening the SDP relaxation}
Next, we investigate several options to improve the SDP 
relaxation~\eqref{eq:sdpCutwidth} of the MCP.
 
\subsubsection{3-dicycle equations} 
In the basic SDP relaxation~\eqref{eq:sdpCutwidth}, no specific information  
about the fact that $x$ should represent a permutation is used. One possible 
way to include such information is to model the transitivity by so-called 
3-dicycle equations, as it is done by 
\cite{buchheim2010}.
These 3-dicycle equations make 
 sure that if $i$ is before $j$ and
 $j$ is before $k$, then $i$ must be before $k$. For a vector  $x = (x_{ij})_{1 
 \leq i < j \leq n}$, they can be written as
 \begin{align}
 x_{ik}-x_{ij}x_{ik}-x_{ik}x_{jk}+x_{ij}x_{jk}=0 & \quad \forall i < j <
 k \label{eq:3dicycle},
 \end{align}
 so we can include
  \begin{align}
 x_{ik}-X_{ij,ik}-X_{ik,jk}+X_{ij,jk}=0 & \quad \forall  i < j <
 k \label{eq:3dicycleSDP}
 \end{align}
 as additional constraints into our SDP relaxation~\eqref{eq:sdpCutwidth} of 
 the MCP. If we add all possible constraints of the 
 form~\eqref{eq:3dicycleSDP}, we include $\binom{n}{3}$ 
 equality constraints. Note that adding also the non-convex constraint $\rank(\bar{X}) = 1$ 
 to this SDP relaxation, the 
 optimal objective function value becomes $\CWG$.
 This is also the case if we  
 add integer conditions for~$X$.
 Observe, that 
 \cite{buchheim2010} show that if the variables 
 are 
 binary, the expression on the left hand-side of~\eqref{eq:3dicycle} is always 
 non-negative. So in this case, one needs to consider only the inequality $\le 0$. 
 However, this does not hold 
 anymore for our SDP relaxation, as the variables are not necessarily binary.

\subsubsection{Triangle inequalities}
Another way to strengthen the SDP relaxation~\eqref{eq:sdpCutwidth} is  
adding the triangle inequalities
\begin{subequations}
	\label{eq:triangle}
	\begin{align}
	0 &\leq X_{ij,k\ell} && \forall i<j, k<\ell \label{eq:triangle1}\\
		X_{ij,k\ell} &\leq X_{ij,ij} && \forall i<j, k<\ell 
		\label{eq:triangle2} \\	
		X_{ij,ij} + X_{k\ell,k\ell} &\leq 1 + X_{ij,k\ell} &&\forall i<j, 
		k<\ell 
		\label{eq:triangle3} \\	
		X_{ij,k\ell} + X_{uv,k\ell} &\leq X_{k\ell,k\ell} + X_{ij,uv} && 
		\forall i<j, 
		k<\ell, u<v \label{eq:triangle4}\\	
		 X_{ij,ij} + 
		 X_{k\ell,k\ell} + X_{uv,uv}
		 &\leq 
		 1 + X_{ij,k\ell} + X_{ij,uv} + X_{k\ell,uv}
		   &&\forall  i<j, k<\ell, u<v.
		\label{eq:triangle5}
	\end{align}
\end{subequations}
The inequalities~\eqref{eq:triangle1} and~\eqref{eq:triangle2} were 
introduced by 
\cite{LovaszSchrijverHierarchy} and 
the constraints~\eqref{eq:triangle3},~\eqref{eq:triangle4} and 
\eqref{eq:triangle5}
originate from 
\cite{GruberRendl}.
As a consequence,~\eqref{eq:triangle1},~\eqref{eq:triangle2} 
and~\eqref{eq:triangle3} each yield $\binom{n}{2}^2$ potential new constraints. 
Both~\eqref{eq:triangle4} and~\eqref{eq:triangle5} offer the option of 
including $\binom{n}{2}^3$ inequalities.

It can be checked easily that all the inequalities of~\eqref{eq:triangle} are 
satisfied for each matrix $X = xx^\top$ for any $x \in  \LOn \cap 
\set{0,1}^{\binom{n}{2}}$, so for any vector $x$ that represents a permutation. 
Thus, the inequalities~\eqref{eq:triangle} are valid inequalities for the SDP 
relaxation~\eqref{eq:sdpCutwidth} for the MCP. 
Next, we give an alternative reasoning that the 
inequalities~\eqref{eq:triangle} are 
satisfied.

\subsubsection{Inequalities obtained from the squared linear ordering polytope}
\cite{AARW} 
introduced so-called exact subgraph 
constraints (ESC), which were later 
computationally exploited for 
the stable set problem, the Max-Cut problem and the coloring problem by 
\cite{GaarRendlIPCO,GaarRendlFull}. 

For the stable set problem the ESCs are defined in the following way.
Let $G = (V,E)$ be a graph with vertex set $V = \{1, 2, \dots, n\}$. The 
squared stable set polytope 
$\STAB^{2}(G)$ of $G$ is defined as
\begin{align*}
\STAB^{2}(G) = \conv\left\{ss^{T}: s \in \{0,1\}^{n}, 
s_{i}s_{j} = 0 \quad \forall 
\{i,j\} \in E \right\},
\end{align*}
and the ESCs for the stable set problem (ESCSS) are 
\begin{align*}
X_I \in \STAB^2(G_I), 
\end{align*} 
where $I \subseteq 
V$, $G_I$ is the induced subgraph of $G$ on the vertices $I$, $X$ is the matrix 
variable of an SDP relaxation of the stable set problem, and $X_I$ is the 
submatrix of $X$ that corresponds to $G_I$.
As detailed \revision{by}~\cite{GaarSiebenhoferWiegeleStabBaB}, 
the ESCSS are equivalent 
to 
\begin{align*}
X_I \in \STAB^2(G^{0}_{|I|}), 
\end{align*} 
where the graph $G^{0}_k = (V^{0}_k,E^{0}_k)$ is given as  
$V^{0}_k = \{1, 2, \dots, k\}$ and $E^{0}_k = \emptyset$.
This implies that only the squared stable set polytope for a graph with $k$ 
vertices and no edges needs to be considered.

Furthermore, 
\cite{GaarSiebenhoferWiegeleStabBaB} describe that 
the ESCSS can be 
represented by 
inequalities for $X_I$, and therefore as inequalities for $X$. 
In~\cite{GaarVersionsESH}, these inequalities that represent 
the ESCSSs
are stated explicitly for subgraphs with $2$ and $3$ vertices. 
In fact, the triangle inequalities~\eqref{eq:triangle} are exactly the 
inequalities representing the ESCSS for subgraphs 
of order 
$2$  
(\eqref{eq:triangle1}, \eqref{eq:triangle2} and \eqref{eq:triangle3}) and for 
subgraphs of order $3$ 
(\eqref{eq:triangle4} and \eqref{eq:triangle5}). 
When 
deriving the inequalities that represent the ESCSS for $G^{0}_k$, 
any binary vector of dimension $k$ is feasible for the corresponding stable set 
problem in $G^{0}_k$. As a 
consequence, $X = xx^T$ is in $\STAB^2(G^{0}_{k})$ for any $x\in \{0,1\}^k$.

For the MCP, only specific (and not all) binary vectors of dimension 
$\binom{n}{2}$ induce a 
permutation of the $n$ vertices. Thus, it makes sense to deduce the ESCs
specifically for the linear ordering problem as they might be more restrictive.
Towards this end, we introduce the quadratic linear ordering polytope 
\begin{align*}
\LOTwon = \conv \set{ \chi^\pi {\chi^\pi}^\top \colon \pi \in \Pn }, 
\end{align*}
which is a subset of $\R^{\binom{n}{2} \times \binom{n}{2}}$. 
From our considerations above, it follows that
\begin{align*}
\LOTwon \subseteq \STAB^2(G^{0}_{\binom{n}{2}})
\end{align*}
holds for all $n \in \mathbb{N}$.

The vertices of the polytope $\LOTwon$ can easily be enumerated.
With the help of the software PORTA \revision{by}~\cite{PORTA} it is possible 
to 
determine the 
inequalities that represent $\LOTwon$ for small values of $n$. In 
particular, 
any matrix $X$ 
is 
in $\LOTwon$ for $n=3$ if and only if it is symmetric and all facet defining 
inequalities and equations 
\begin{subequations}
	\label{eq:LO2_n3}
	\begin{align}	
	0 &\leq X_{ij,jk} && \forall i<j<k 
	\label{eq:LO2_n3_1}\\
	X_{ij,ik}\leq X_{ij,ij}, \quad
	X_{ij,ik}\leq X_{ik,ik}, \quad
	X_{ik,jk}&\leq X_{ik,ik}, \quad
	X_{ik,jk}\leq X_{jk,jk}
	&& \forall i<j<k 
	\label{eq:LO2_n3_2}\\
	X_{ij,ij} + X_{jk,jk}&\leq 1 + X_{ij,jk}&& \forall i<j<k 
	\label{eq:LO2_n3_3} \\	
	X_{ik,ik} -X_{ij,ik} - X_{ik,jk} + X_{ij,jk} &= 0  && \forall i<j < k 
	\label{eq:LO2_n3_0}
	\end{align}
\end{subequations}
are satisfied. Note, that~\eqref{eq:LO2_n3_1},~\eqref{eq:LO2_n3_3} 
and~\eqref{eq:LO2_n3_0} each yield $\binom{n}{3}$ potential inequalities 
and~\eqref{eq:LO2_n3_2}
offers the option of 
including $4\binom{n}{3}$ constraints.
It is easy to see that 
 \eqref{eq:LO2_n3_0} coincides with the 3-dicycle equations already considered 
as~\eqref{eq:3dicycleSDP}. 
Furthermore,
\eqref{eq:LO2_n3_1}, \eqref{eq:LO2_n3_2} and 
\eqref{eq:LO2_n3_3} are a subset of 
the triangle inequalities  
\eqref{eq:triangle1},  \eqref{eq:triangle2} and  \eqref{eq:triangle3}, 
respectively. 
It turns out that the following holds.
\begin{lemma}
	\label{lem:ESCstronger}
	If a symmetric $X$ satisfies~\eqref{eq:LO2_n3}, then $X$ also
	fulfills all inequalities of~\eqref{eq:triangle} whenever 
	$|\{i,j,k,\ell,u,v\}| \leq 3$ holds.
\end{lemma}
\begin{proof}
	Assume a symmetric $X$ satisfies~\eqref{eq:LO2_n3}.
	
	We first show that then $X$ fulfills~\eqref{eq:triangle2} whenever 
	$|\{i,j,k,\ell,u,v\}| \leq 3$ holds. 
	Clearly $|\{i,j,k,\ell,u,v\}| \geq 2$ always holds, 
	and~\eqref{eq:triangle2} is trivial if $|\{i,j,k,\ell,u,v\}| = 2$. 
	So let $\{i,j,k,\ell,u,v\} = \{i',j',k'\}$ with $i' < j' < k'$. 
	We have to show that $X$ fulfills
	\begin{subequations}
		\label{eq:LO2_n3_MainDom}
		\begin{align}	
		X_{i'j',i'k'}&\leq X_{i'j',i'j'} \label{eq:LO2_n3_MainDom_1}\\
		X_{i'j',j'k'}&\leq X_{i'j',i'j'} \label{eq:LO2_n3_MainDom_2}\\
		X_{i'k',i'j'}&\leq X_{i'k',i'k'} \label{eq:LO2_n3_MainDom_3}\\
		X_{i'k',j'k'}&\leq X_{i'k',i'k'} \label{eq:LO2_n3_MainDom_4}\\
		X_{j'k',i'j'}&\leq X_{j'k',j'k'} \label{eq:LO2_n3_MainDom_5}\\
		X_{j'k',i'k'}&\leq X_{j'k',j'k'} \label{eq:LO2_n3_MainDom_6}.
		\end{align}
	\end{subequations}
Clearly, the inequalities~\eqref{eq:LO2_n3_MainDom_1}, 
\eqref{eq:LO2_n3_MainDom_3}, \eqref{eq:LO2_n3_MainDom_4} 
and~\eqref{eq:LO2_n3_MainDom_6} are fulfilled because of~\eqref{eq:LO2_n3_2}.

Furthermore, $X_{i'k',i'k'} = X_{i'j',i'k'} + X_{i'k',j'k'} - X_{i'j',j'k'}$ 
because of~\eqref{eq:LO2_n3_0} and $X_{i'k',j'k'}\leq X_{i'k',i'k'}$ due 
to~\eqref{eq:LO2_n3_2}. Thus, 
$X_{i'k',j'k'} \leq X_{i'j',i'k'} + X_{i'k',j'k'} - X_{i'j',j'k'}$ 
holds, which implies that
$X_{i'j',j'k'}\leq X_{i'j',i'k'}$ holds. Together with $X_{i'j',i'k'}\leq 
X_{i'j',i'j'}$ because of~\eqref{eq:LO2_n3_2}, this implies 
$X_{i'j',j'k'}\leq X_{i'j',i'j'}$, and so~\eqref{eq:LO2_n3_MainDom_2} holds. 

Analogously it can be shown that~\eqref{eq:LO2_n3_MainDom_5} holds with the 
help of~\eqref{eq:LO2_n3_0} and~\eqref{eq:LO2_n3_2}. As a result, $X$ it 
fulfills~\eqref{eq:triangle2} whenever 
$|\{i,j,k,\ell,u,v\}| \leq 3$ holds.
In a similar fashion, 
\begin{itemize}
	\setlength\itemsep{0em}
	\item \eqref{eq:LO2_n3_0},~\eqref{eq:LO2_n3_1} and~\eqref{eq:triangle2} 
	imply~\eqref{eq:triangle1},
	\item \eqref{eq:LO2_n3_0},~\eqref{eq:LO2_n3_2} and~\eqref{eq:LO2_n3_3}
	imply~\eqref{eq:triangle3},
	\item \eqref{eq:LO2_n3_0} and~\eqref{eq:triangle2} 
	imply~\eqref{eq:triangle4}, and 
	\item \eqref{eq:LO2_n3_0},~\eqref{eq:triangle1} and~\eqref{eq:triangle3}
	imply~\eqref{eq:triangle3}
\end{itemize}
whenever $|\{i,j,k,\ell,u,v\}| \leq 3$ holds. As a consequence, all 
inequalities 
of~\eqref{eq:triangle} are satisfied in this case.	  
\end{proof}

Note that Lemma~\ref{lem:ESCstronger} is indeed no surprise: With the intuition 
of the ESCs for the stable set problem, the triangle 
inequalities~\eqref{eq:triangle} (which represent the membership to 
$\STAB(G^{0}_3)$) must be satisfied for any $3 \times 3$ 
submatrix 
of $X$, as the only condition is that the corresponding submatrix must be 
formed by binary vectors.
Instead, the constraints~\eqref{eq:LO2_n3} (which represent the membership to 
$\LOTwo{3}$) are only valid for $3 \times 3$ 
submatrices of $X$ whose indices that correspond to the three pairs $ij$, $ik$, 
$jk$ for any $1 \leq i<j<k \leq n$. Thus, the inequalities~\eqref{eq:LO2_n3} 
capture more structure, but are valid for fewer submatrices.

With the help of $\LOTwo{4}$ we were able to find the next valid inequalities, 
which are (to the best of the knowledge of the authors) not known as valid 
inequalities for the linear ordering problem so far.
\begin{lemma}
	The inequalities
	\begin{align}\label{eq:LOtwofour}
	 x_{i\ell } +X_{ik,jk} +X_{jk,j\ell } &\leq x_{ik} +x_{j\ell 
	} 
	+X_{ij,k\ell } +X_{i\ell ,jk} \qquad \forall i<j<k<\ell
	\end{align}
	are valid inequalities for the SDP 
	relaxation~\eqref{eq:sdpCutwidth} for the MCP.
\end{lemma}
\begin{proof}
	In order to prove this lemma, it is enough to show 
	that for all $i<j<k<\ell$ and for all $X = 
	\chi^\pi 
	{\chi^\pi}^\top$ for $\pi \in 
	\Pn$  the inequality
	\begin{align}\label{eq:LOtwofour2}
\chi^\pi_{i\ell } +\chi^\pi_{ik}\chi^\pi_{jk} +\chi^\pi_{jk}\chi^\pi_{j\ell 
} &\leq \chi^\pi_{ik} + \chi^\pi_{j\ell 
} 
+\chi^\pi_{ij}\chi^\pi_{k\ell } +\chi^\pi_{i\ell}\chi^\pi_{jk} 
\end{align}
	is satisfied.
	To do so, we distinguish two cases for fixed  $i<j<k<\ell$ 
	and fixed 
	$\pi$.
	
	If $\pi(j) < \pi(k)$, then $\chi^\pi_{jk} = 1$ and~\eqref{eq:LOtwofour2} 
	simplifies to 
	$0 \leq \chi^\pi_{ij}\chi^\pi_{k\ell }$, which is clearly satisfied.
	
	If $\pi(k) < \pi(j)$, then $\chi^\pi_{jk} = 0$ and 
	therefore~\eqref{eq:LOtwofour2} simplifies to
	\begin{align}\label{eq:LOtwofour3}
\chi^\pi_{i\ell }  &\leq \chi^\pi_{ik} + \chi^\pi_{j\ell } 
+\chi^\pi_{ij}\chi^\pi_{k\ell }.
\end{align}
	As this inequality is surely satisfied if $\chi^\pi_{i\ell } = 0$, we only 
	have to investigate $\chi^\pi_{i\ell } = 1$, so $\pi(i) < \pi(\ell)$. 
	Assume $\chi^\pi_{ik} + \chi^\pi_{j\ell} = 0$, then $\pi(k) < \pi(i)$ and 
	$\pi(\ell) < \pi(j)$. With all other relations we have this implies that 
	$\pi(k) < \pi(i) < \pi(\ell) < \pi(j)$ holds. Thus,
	$\chi^\pi_{ij}\chi^\pi_{k\ell } = 1$ holds under our assumption, which 
	implies that the right 
	hand-side of~\eqref{eq:LOtwofour3} is at least one.  
	Therefore,~\eqref{eq:LOtwofour2} is fulfilled in this case.
	
	As a result, in any case~\eqref{eq:LOtwofour2} is satisfied, which 
	finishes the proof.
\end{proof}

\subsubsection{Liftings of inequalities from the linear ordering polytope}

\cite{buchheim2010} derived their quadratic 3-dicycle 
equations~\eqref{eq:3dicycle} as an alternative to the linear 
3-dicycle inequalities
\begin{align}
\label{eq:lin3dicycle}
	0 \leq x_{ij} + x_{jk} - x_{ik} \leq 1 \qquad &\forall i<j<k.
\end{align}

One could also use the standard reformulation linearization technique (RLT), 
of which the foundations were laid by 
\cite{adams1986tight}. 
Using this approach, we can multiply the inequalities~\eqref{eq:lin3dicycle} by 
$x_{uv}$ and $(1-x_{uv})$ for every pair $(u,v)$ with 
$u<v$, and then replace products of the form $x_{ij}x_{k\ell}$ by 
$X_{ij,k\ell}$. In this way we obtain the set of valid inequalities
\begin{subequations}
	\label{eq:lin3d}
\begin{alignat}{3}
	0  \leq~ &   X_{ij,uv} + X_{jk,uv} - X_{ik,uv} &&  &&\forall 
	i<j<k, u<v\label{eq:lin3dy_xLHS}\\
		 &   X_{ij,uv} + X_{jk,uv} - X_{ik,uv} &&\leq x_{uv}  
		 &&\forall 
	i<j<k, u<v\label{eq:lin3dy_xRHS}\\
		0  \leq~ &  x_{ij} + x_{jk} - x_{ik} - X_{ij,uv} - X_{jk,uv} + 
		X_{ik,uv} 
		&&  &&\forall i<j<k, u<v\label{eq:lin3dy_1-xLHS}\\
			 &   x_{ij} + x_{jk} - x_{ik} - X_{ij,uv} - 
			X_{jk,uv} + X_{ik,uv} 
		&&\leq 1- x_{uv} \quad &&\forall i<j<k, u<v,\label{eq:lin3dy_1-xRHS}
\end{alignat}
\end{subequations}
which yields $4\binom{n}{3}\binom{n}{2}$ potential additional constraints. 
However, some of them are already implied by the inequalities previously 
introduced. 
Indeed, 
\begin{itemize}
	\setlength\itemsep{0em}
	\item \eqref{eq:lin3dy_xLHS} for 
	$(u,v)=(i,j)$ and $(u,v)=(j,k)$
	is implied by~\eqref{eq:triangle1} and ~\eqref{eq:triangle2},  
	\item \eqref{eq:lin3dy_xLHS} for $(u,v) = (i,k)$ is 
	implied by~\eqref{eq:triangle1} and~\eqref{eq:3dicycleSDP},
		\item \eqref{eq:lin3dy_xRHS} for 
	$(u,v)=(i,j)$ and $(u,v)=(j,k)$
	is implied by~\eqref{eq:triangle2} and ~\eqref{eq:3dicycleSDP},
\item \eqref{eq:lin3dy_xRHS} for $(u,v) = (i,k)$  is 
implied by~\eqref{eq:triangle2},
\item \eqref{eq:lin3dy_1-xLHS} for $(u,v)=(i,j)$ and $(u,v)=(j,k)$ 
is implied by~\eqref{eq:triangle2} and~\eqref{eq:3dicycleSDP},
\item \eqref{eq:lin3dy_1-xLHS} for $(u,v) = (i,k)$ is 
implied by~\eqref{eq:triangle2},
\item \eqref{eq:lin3dy_1-xRHS} for $(u,v)=(i,j)$ and $(u,v)=(j,k)$ 
is implied by~\eqref{eq:triangle2} and~\eqref{eq:triangle3}, and 
\item \eqref{eq:lin3dy_1-xRHS} for $(u,v) = (i,k)$ is 
implied by~\eqref{eq:triangle3} and~\eqref{eq:3dicycleSDP}.
\end{itemize}
Thus, at least one of $u$ and $v$ needs to be different to $i$, $j$ and $k$ 
such that an inequality of~\eqref{eq:lin3d} has the potential to bring new 
information.

\section{Algorithms}\label{sec:algo} 
In this section we describe in detail our algorithm that utilizes 
our new SDP relaxation for the MCP and its strengthenings derived in the last 
section.  
Furthermore, we present a new upper bound obtained by a heuristic 
utilizing the optimizer of the SDP relaxation.

\subsection{Algorithm for computing our new lower bound}\label{sec:algoLB} 

Adding all the previously described inequality and equality constraints at once 
to the 
basic SDP 
relaxation~\eqref{SDP_init} would render it too computationally expensive to 
solve. Therefore, a cutting-plane approach is used to obtain a tight lower 
bound for the MCP.

\subsubsection{Separating constraints}
The total number of possible inequality and equality constraints to add to the 
basic SDP 
relaxation~\eqref{SDP_init} is $\mathcal{O}(n^6)$. While 
it is possible to exhaustively enumerate them and keep the ones with the 
largest violation, it would take a long time and does not guarantee the best 
tightening of the relaxation. A heuristic to find violated constraints is 
thus preferred. 

The algorithm used 
to find a single violated constraint for a given solution $\bar{X}$ of the SDP 
relaxation~\eqref{SDP_init} is a simulated annealing heuristic, in 
which the current solution is represented by a tuple of indices 
$((i,j,k,\ell),(u,v),(q,w))$, with $i<j<k<\ell$, $u<v$ and $q < w$, coupled 
with the constraint 
type. 
The possible constraint types are 
the 3-dicycle equations~\eqref{eq:3dicycleSDP}; 
the triangle inequalities~\eqref{eq:triangle1}, \eqref{eq:triangle2}, 
\eqref{eq:triangle3}, \eqref{eq:triangle4}, \eqref{eq:triangle5}; 
the inequality from the squared linear ordering polytope of order 
four~\eqref{eq:LOtwofour}, 
and the inequalities obtained from lifting the linear ordering 
polytope~\eqref{eq:lin3dy_xLHS}, \eqref{eq:lin3dy_1-xLHS}, 
\eqref{eq:lin3dy_xRHS}, \eqref{eq:lin3dy_1-xRHS} and are given in the array 
\texttt{constraintsToTest}.
Note that the tuple of indices is defined in such a way that the required 
indices of all possible constraints can be extracted from it. For example, the 
three pairs of indices of~\eqref{eq:triangle5} can be extracted as $(i,j)$, 
$(u,v)$ and $(q,w)$ from the current solution.

Whenever a current solution is given, a neighbor solutions are obtained 
with \textit{random\_neighbor\_indices}() 
by randomly replacing one index from the tuple of indices of the current 
solution 
before ordering it 
again. 
The violation of a solution (i.e., a fixed type of constraint for a fixed tuple 
of indices) 
for the solution $\bar{X}$ of the SDP relaxation~\eqref{SDP_init} can be 
computed via \textit{violation}(), and is positive if the inequality or 
equality constraint is violated.
The pseudocode of this simulated annealing heuristic can be found in 
Algorithm~\ref{alg:SA_cuts}.


\RestyleAlgo{ruled}

\begin{algorithm}[tbhp]
	\DontPrintSemicolon

	\SetKwInput{KwInput}{Input}                

	\SetKwInput{KwOutput}{Output}              
	\SetKwInput{KwParameters}{Parameters}

	\KwParameters{$T_{\text{init}}$, $f_t$,  \revision{\texttt{maxIterSep}}, 
	\texttt{maxLenPlateau}}
	\KwInput{solution $\bar{X}$ of the SDP relaxation~\eqref{SDP_init},  array 
	of types of constraints 
	\texttt{constraintsToTest}}
	
	iter $\longleftarrow$ 0\;
	lenPlateau $\longleftarrow$ 0\;
	$(i,j,k,\ell) \longleftarrow$ 
	\textit{random}($\set{1, \ldots, n}$)  with 
	$i<j<k<\ell$\;
	$(u,v)$ $\longleftarrow$ 
	\textit{random}($\set{1, \ldots, n}$)  with 
	$u<v$\;	
	$(q,w)$ $\longleftarrow$ 
	\textit{random}($\set{1, \ldots, n}$)  with 
	$q<w$\;		
	indices $\longleftarrow ((i,j,k,\ell), (u,v), (q,w))$\;
	currentSolution $\longleftarrow$ (indices, None)\;
	currentViolation $\longleftarrow$ $-\infty$\;
	bestSolution $\longleftarrow$ currentSolution\;
	bestViolation $\longleftarrow$ currentViolation\;
	$T \longleftarrow T_{\text{init}}$\;
	\While {\upshape{lenPlateau} < \upshape{\texttt{maxLenPlateau}} and 
	\upshape{iter} < \upshape{ \revision{\texttt{maxIterSep}}}}
	{
		
		iter $\longleftarrow$ iter + 1\;
		lenPlateau $\longleftarrow$ lenPlateau + 1\;
		neighborIndices $\longleftarrow$ 
		\textit{random\_neighbor\_indices}(currentSolution)\;
		\For{\upshape{inequality} in \upshape{\texttt{constraintsToTest}}}{
			neighborSolution = (neighborIndices, inequality)\;
			neighborViolation  $\longleftarrow$ 
			\textit{violation}(neighborSolution,  
			$\bar{X}$)\;
			$\Delta  \longleftarrow$  neighborViolation $-$ currentViolation 
			\;
			\eIf {$\Delta > 0$}{
				currentSolution $\longleftarrow$ neighborSolution\;
				currentViolation $\longleftarrow$ neighborViolation\;}
			{\If{$\text{random}([0,1])$ < $e^{\frac{\Delta}{T}}$}
				{currentSolution $\longleftarrow$ neighborSolution\;
				currentViolation $\longleftarrow$ neighborViolation\;
				$ T \longleftarrow T \cdot f_t$}
		}
	
		\If{\upshape{currentViolation} > \upshape{bestViolation}}
			{bestSolution  $\longleftarrow$ currentSolution\;
				bestViolation $\longleftarrow$ currentViolation\;
				lenPlateau $\longleftarrow$ 0
				}			
	
		}		
}
\Return bestSolution, bestViolation
	\caption{Simulated annealing to separate constraints}\label{alg:SA_cuts}
\end{algorithm}

\subsubsection{Outline of the overall algorithm}
After detailing how to find violated inequality and equality constraints with 
simulated annealing in 
Algorithm~\ref{alg:SA_cuts}, we are now able to describe our main algorithm 
to 
derive a lower bound for the MCP.

Our algorithm starts by solving the 
basic SDP relaxation~\eqref{SDP_init}, providing a first lower bound 
$\alpha_{\text{init}}$ and the associated solution $\bar{X}_{\text{init}}$.  
Some constraints violated by the current solution $\bar{X}_{\text{init}}$, are 
then determined with Algorithm~\ref{alg:SA_cuts}. In particular, 
at each iteration, Algorithm~\ref{alg:SA_cuts} is run 
$2\texttt{numCuts}$ times, and 
the \texttt{numCuts} most violated constraints are added. 
These violated constraints are then added to the SDP, which is solved 
again to obtain a new improved lower bound $\alpha$ and a new current solution 
$\bar{X}$. 
These steps are repeated for a fixed number of 
iterations   \revision{\texttt{maxIterCP}}, or until the improvement of the 
bound does not 
reach a certain 
threshold.

To reduce the computational effort, at each iteration the constraints that 
seem to be not necessary for obtaining the bound with the SDP are removed. To 
do so, the mean $\gamma_{\text{mean}}$ 
of all absolute values of all dual values $\gamma$ associated with each added 
inequality is computed. 
The constraints that are associated with a dual value  $|\gamma| < 
0.01\gamma_{\text{mean}}$ are then removed.
The pseudocode of our algorithm can be found in 
Algorithm~\ref{alg:CW}.

\RestyleAlgo{ruled}

\begin{algorithm}[tbp]
	
	\DontPrintSemicolon
	\SetKwInput{KwInput}{Input}                
	\SetKwInput{KwOutput}{Output}              
	\SetKwInput{KwParameters}{Parameters} 
	
	\KwParameters{ \revision{\texttt{maxIterCP}}, \texttt{improvementMin}, 
	\texttt{numCuts}, 
	\texttt{minViolation}}
	\KwInput{adjacency matrix $A$, array of types of constraints 
	\texttt{constraintsToTest}}
	solve the SDP relaxation~\eqref{SDP_init} to obtain  
	 $\bar{X}_{\text{init}}$ and the lower bound 
	$\alpha_{\text{init}}$\;
	$\bar{X} \longleftarrow \bar{X}_{\text{init}}$\;
	$\alpha \longleftarrow \alpha_{\text{init}}$\;
	iter $ \longleftarrow 0$\;
	improvement $ \longleftarrow +\infty$\;
	addedCuts = $\emptyset$\;
	\While{\upshape{iter} < \upshape{ \revision{\texttt{maxIterCP}}} and 
		\upshape{improvement} $>$ \upshape{\texttt{improvementMin}} }{
		iter $\longleftarrow$ iter + 1\;
		update \texttt{constraintsToTest}\;
		\For{\upshape{iterTemp in $\{1, \dots, 2\texttt{numCuts}\}$}}{
			(cutTemp,violationTemp) $\longleftarrow$
			Algorithm~\ref{alg:SA_cuts} with input $\bar{X}$, 
			\texttt{constraintsToTest}\;
			\If{\upshape{violationTemp} > \upshape{\texttt{minViolation}}}{
				add cutTemp to addedCuts}
		}
		addedCuts $\longleftarrow$ \texttt{numCuts} most violated cuts from 
		addedCuts\;
		solve~\eqref{SDP_init} with all inequalities from addedCuts to obtain 
		$\bar{X}_{\text{new}}$, $\alpha_{\text{new}}$, and the dual values 	
		$\gamma$ associated with each added inequality from 
		addedCuts\;
		$\gamma_{\text{mean}} \longleftarrow$ mean of the absolute values of 
		all dual values $\gamma$\;
		remove constraints with dual value  $|\gamma| < 
		0.01\gamma_{\text{mean}}$ from 
		addedCuts\;
		improvement $\longleftarrow \alpha_{\text{new}} - \alpha$\;
		$\bar{X} \longleftarrow \bar{X}_{\text{new}}$\;
		$\alpha \longleftarrow \alpha_{\text{new}}$\;
	}
	\Return $\alpha$, $\bar{X}$
	\caption{Cutting-plane algorithm to obtain a lower bound for the MCP 
	}\label{alg:CW}
\end{algorithm}


The update of \texttt{constraintsToTest} is done in the following way to 
improve the efficiency.
Empirically, the constraints added in the first two iterations are mostly
3-dicycle equations~\eqref{eq:3dicycleSDP}. This is consistent with the fact 
that they are the ones  
adding the most to the structure of the problem. From the second or third 
iteration on (depending mostly on the size of the instances), the triangle 
inequalities~\eqref{eq:triangle} represent the most part of the added cuts. 
Thus, in the first two 
iterations of our algorithm only violated 3-dicycle equations are added. 
Triangle inequalities are then added to our pool of potential constraints 
\texttt{constraintsToTest} for the 
third and fourth iterations, and all constraints are considered for the 
remaining iterations of the algorithm.

\subsection{Our new heuristic for computing an upper bound}

We can utilize the SDP relaxation~\eqref{eq:sdpCutwidth} of the MCP not only
 to derive lower bounds as seen in Section~\ref{sec:algoLB}, but also to obtain
  upper bounds on the cutwidth. In particular, our aim is to derive a feasible 
  solution of the 
  MCP, and
   thus an upper bound, from the SDP solution $\bar{X}$ returned by
    Algorithm~\ref{alg:CW}. 

For that purpose, the first column $x$ of $\bar{X}$ is considered, without
 its first element (equal to $1$ by definition of $\bar{X}$). 
So $x = (x_{ij})_{1\leq i < j \leq n}$ is a vector of size $\binom{n}{2}$,
 with entries between $0$ and $1$. For any $i \in V= \set{1, 2, \dots, n}$, we 
 compute the
  relative position $p_i$ of vertex $i$ as $$p_i =
   \sum\limits_{\substack{j \in V \\ j>i}} (1-x_{ij}) +
    \sum\limits_{\substack{j \in V \\j<i}} x_{ji}.$$
For any vector $x \in \LOn \cap \set{0,1}^{\binom{n}{2}}$ encoding
 a linear ordering $\pi$, $p_i$ is equal to the number of vertices before the
  vertex $i$ in $\pi$, i.e., to the position of $i$ in the linear ordering. In 
  the
   general case of a vector $x$ obtained from the SDP
    relaxation~\eqref{SDP_init} (with or without added cuts), and not 
    necessarily
     feasible for the MCP, the  $p_i$ are first sorted in ascending order:
      let $(i_1, \ldots, i_n)$ be an ordering of $V$
       such that $p_{i_{1}} \leq p_{i_{2}} \leq \ldots \leq p_{i_{n}}$.
A linear ordering $\pi$ can then be retrieved by assigning to each vertex $i$ 
the
 position of $p_i$ in the sorted list, so $\pi(i_1) = 1, \pi(i_2) = 2,
  \ldots, \pi(i_n) = n$.
The pseudocode of this algorithm can be found in Algorithm~\ref{alg:UB}.
 
This \revision{algorithm} provides a feasible linear ordering for the MCP, which we obtained directly
 from a feasible SDP solution. Thus, the cutwidth of this ordering is an
  upper bound for the MCP. This upper bound is then improved by running a
   simulated annealing heuristic, that uses the obtained feasible solution as
    starting solution. \revision{ In this heuristic, random neighbors of the 
    current feasible ordering are obtained by swapping at random two vertices 
    of the ordering, using the function \textit{random\_neighbor\_ordering}(). 
    This function takes as argument the current feasible solution, as well as 
    the list of already visited orderings, to ensure that no ordering 
    is considered more than once in the algorithm. The cutwidth of an 
    ordering in a given graph is computed via the function \textit{cutwidth}(). The pseudocode 
    of the heuristic can be found in Algorithm~\ref{alg:SA_UB}.}

\RestyleAlgo{ruled}
\begin{algorithm}[tbph]
	
	\DontPrintSemicolon
	\SetKwInput{KwInput}{Input}                
		
	\KwInput{$V= \set{1, 2, \dots, n}$, matrix $\bar{X}$ returned by 
	Algorithm~\ref{alg:CW}}
	$x =  (x_{ij})_{1\leq i < j \leq n} \longleftarrow $ first row of $\bar{X}$ 
	without the first entry \;
	\ForEach{$i \in V$}{
		$p_i \longleftarrow \sum\limits_{\substack{j \in V\\ j>i}} (1-x_{ij}) + \sum\limits_{\substack{j \in V\\j<i}} x_{ji}$ \;
	}
	$(i_1, \ldots, i_n) \longleftarrow \text{ordering of } V$
	 such that $p_{i_{1}} \leq p_{i_{2}} \leq \ldots \leq p_{i_{n}}$\;
	\ForEach{$k \in V$}{
		$\pi(i_k) \longleftarrow k$ \;
	}
	\Return $\pi$
	\caption{Obtaining a linear ordering from a feasible SDP 
	solution}\label{alg:UB}
\end{algorithm}

\RestyleAlgo{ruled}
\begin{algorithm}[tbph]
	
	\DontPrintSemicolon
	\SetKwInput{KwInput}{Input}                
	\SetKwInput{KwParameters}{Parameters} 
	
	\revision{
	\KwParameters{$T_{\text{init}}^{\scriptscriptstyle{UB}}$, 
	$f_t^{\scriptscriptstyle{UB}}$, \texttt{maxIterUB}, 
	\texttt{maxLenPlateauUB}}
	\KwInput{graph $G$, linear ordering $\pi$}
		
	iter $\longleftarrow$ 0\;
	lenPlateau $\longleftarrow$ 0\;
	$T \longleftarrow T_{\text{init}}^{\scriptscriptstyle{UB}}$\;
	
	currentOrdering $\longleftarrow$ $\pi$\;
	currentCutwidth $\longleftarrow$ \textit{cutwidth}($G$, currentOrdering)\;
	bestOrdering $\longleftarrow$ currentOrdering\;
	bestCutwidth $\longleftarrow$ currentCutwidth\;
	visitedOrderings $\longleftarrow$ $[\mbox{currentOrdering}]$\;
	
	\While {\upshape{lenPlateau} < \upshape{\texttt{maxLenPlateauUB}} and 
		\upshape{iter} < \upshape{\texttt{maxIterUB}}}
	{
		iter $\longleftarrow$ iter + 1\;
		lenPlateau $\longleftarrow$ lenPlateau + 1\;
		neighborOrdering $\longleftarrow$ 
		\textit{random\_neighbor\_ordering}(currentOrdering, 
		visitedOrderings)\;
		append neighborOrdering to visitedOrderings \;
		neighborCutwidth  $\longleftarrow$ 
		\textit{cutwidth}($G$, neighborOrdering)\;
		$\Delta  \longleftarrow$  neighborCutwidth $-$ currentCutwidth 
		\;
		\eIf {$\Delta < 0$}{
			currentOrdering $\longleftarrow$ neighborOrdering\;
			currentCutwidth $\longleftarrow$ neighborCutwidth\;}
		{\If{$\text{random}([0,1])$ < 
		$e^{-\frac{\Delta}{T}}$}
			{currentOrdering $\longleftarrow$ neighborOrdering\;
				currentCutwidth $\longleftarrow$ neighborCutwidth\;
				$ T \longleftarrow T \cdot 
				f_t^{\scriptscriptstyle{UB}}$}
			}
			\If{\upshape{currentCutwidth} < \upshape{bestCutwidth}}
			{bestOrdering  $\longleftarrow$ currentOrdering\;
				bestCutwidth $\longleftarrow$ currentCutwidth\;
				lenPlateau $\longleftarrow$ 0
			}			
			}	
		\Return bestOrdering, bestCutwidth
	\caption{Simulated annealing to improve the upper bound for the MCP  
	given by an initial linear ordering}\label{alg:SA_UB}
	}
	\end{algorithm}

\section{Computational experiments}\label{sec:computations}
In this section we evaluate the quality of the bounds obtained by
our algorithms through numerical results.
We compare the bounds to those obtained by modeling the problem
as MILP.

\subsection{Computational setup}

 Our algorithm was implemented in Python, using the graph library 
 \texttt{networkx}. The SDP relaxations were solved using the \cite{mosek} 
 Optimizer 
 API version~10.
 We use CPLEX \revision{by}~\cite{cplex} version~22.1 to solve the MILP in 
 order to
 compare bounds and computation times of our approach to \revision{a} MILP
 approach from the literature. 
 We ran the tests on an Intel Xeon W-2125 CPU~@~4.00GHz with 
 4~Cores /  8~vP  with 256~GB RAM.
 The code and the instances are available as ancillary files from the arXiv 
 page of this paper\revision{: \href{https://arxiv.org/abs/2301.03900}{arXiv:2301.03900}}.

For the experiments, the parameters introduced in 
\revision{Algorithms~\ref{alg:SA_cuts}, \ref{alg:CW} and~\ref{alg:SA_UB}} 
were set as follows:
$T_{\text{init}} = 0.042$, $f_t = 0.97$, 
$\revision{\texttt{maxIterSep}} = \binom{n}{2}$, 
\texttt{maxLenPlateau} = $\floor{\frac{n}{2}}$, 
$\revision{\texttt{maxIterCP}} = 7$, 
\texttt{improvementMin} = $10^{-2}$, \texttt{numCuts} = $2n²$, and 
\texttt{minViolation} = $10^{-4}$, 
\revision{$T_{\text{init}}^{\scriptscriptstyle{UB}} = 0.25$, 
$f_t^{\scriptscriptstyle{UB}} = 0.99$, \texttt{maxIterUB} =  $\binom{n}{2}$, 
\texttt{maxLenPlateauUB} = $5n$}.

\subsection{Instances}\label{sec:instances}
 \revision{There are several sets of instances commonly used in the literature 
 as benchmark for linear ordering problems. However, almost all of
 these benchmark instances are very sparse. This is due to the fact
 that so far they
 have been tested mainly on MILP-based approaches and these approaches
 are applicable on sparse instances only.
 For evaluating our bounds on instances of varying size and density,
 we generate new instances, namely the Erd\H{o}s-R{\'e}nyi graphs and
 random geometric graphs, as described in Section~\ref{sec:densegraphs}.

 \subsubsection{Benchmark instances in the literature}

 On these benchmarks instances, MILP-based methods are, due to the
 sparsity of the graphs, typically
 performing much better than SDP-based methods, hence we will
 refrain from reporting detailed results. However,
 for completeness we give here the specifics and sources of these instances. 

  The \textit{Small} dataset, 
 introduced for the Bandwidth Minimization Problem by~\cite{marti2008branch} 
 and already used several times for evaluating algorithms for the CMP, 
 consists in $84$ graphs of order between $16$ and $24$, and densities 
 ranging from $0.07$ to $0.21$. 
 
 The \textit{Rome graphs} dataset, introduced in~\cite{di1997experimental}, 
 is  formed of $11534$ graphs whose number of vertices ranges between $10$ 
 and $100$, and densities in average below $0.1$, and rarely above $0.2$.
 
 The \textit{grid} dataset was proposed by~\cite{raspaud1995debruijn} and 
 consists of $81$ rectangular grids of sizes $\mbox{width} \times 
 \mbox{height}$ where $\mbox{width}$ and $\mbox{height}$ are chosen from the 
 set $\{3,6,9,12,15,18,21,24,27\}$. The optimal values of the MCP for these 
 instances are known (see \cite{raspaud2009antibandwidth}).
 
The \textit{Harwell-Boeing} dataset (available at~\cite{boisvert1997matrix}) 
is a set of (mostly) sparse matrices obtained from problems in linear 
systems, least squares, and eigenvalue calculations from diverse 
scientific and engineering disciplines. Graphs are obtained from these 
matrices by considering an edge for every nonzero element of the matrix. For 
experiments on linear ordering problems, a subset of $87$ matrices is 
commonly used, with numbers of vertices ranging from $30$ to $700$, and numbers of
edges from $46$ to $41686$.

The \textit{Small}, \textit{Harwell-Boeing} and \textit{grid} datasets can be 
found at~\cite{optsicom}, and the \textit{Rome graphs} dataset can be 
accessed at~\cite{coudertInstances}. 

Almost all these instances are, as stated, very sparse, with densities most 
of the time much below 20\% and very often even below 10\%. 

\subsubsection{Erd\H{o}s-R{\'e}nyi graphs and random geometric graphs}\label{sec:densegraphs}
It is expected that the MCP on denser graphs is much more challenging to be
solved by MILP-based approaches, whereas our SDP-based approach is little to no 
impacted by density. As such, its relevance for computing lower bounds for 
the MCP becomes much greater for denser instances, on which existing 
algorithms struggle. 
To do experiments on graphs with various densities,
we generated instances according to well established models: 
Erd\H{o}s-R{\'e}nyi graphs and random geometric graphs.

}
 


 An Erdős-Rényi 
 graph $G(n,p)$ is a random graph with $n$ vertices generated by including each 
 possible edge with probability $p$ (independent from the inclusion of other 
 edges). Our set 
 is composed of $30$ Erdős-Rényi graphs $G(n,p)$ with $n \in 
 \{20,30,40,50\}$ and $p \in \{0.3,0.4,0.5,0.6,0.7,0.8,0.9\}$.
 
 A 
 random
 geometric graph $U(n,d)$ is generated by first placing $n$ vertices uniformly 
 at 
 random 
 in the unit cube. Two vertices are then joined by an edge if and only if the 
 Euclidian distance between them is at most $d$, see for 
 example~\cite{johnson1989optimization}. 
 Our set created this way is composed of $45$ graphs with $n 
 \in 
 \{20,30,40,50\}$, the distances $d$ are chosen from the set $\{0.3, 
 0.4, 0.5, 0.6, 0.7, 0.8, 0.9\}$.

\subsection{Numerical results}
\revision{
\subsubsection{Numerical results on the benchmark instances in the literature}
  As already mentioned, on the very sparse graphs used in the
  literature so far, our algorithm is not competitive with the
MILP-based approaches.
For example, the algorithm of~\cite{coudert2016} manages to solve all
instances of the \textit{Small} dataset to optimality with an average
computational time of $3.2$ seconds.  
The results provided by our algorithm have a final gap 
between upper and lower bound ranging between $20\%$ and $63\%$, with a mean of 
$44\%$, obtained with an average time sightly over a minute.
This looks similar for the other instances: we obtain bounds of
moderate quality and in reasonable time but for these sparse
instances, the MILP-based approaches perform much better. We 
therefore refrain from listing the details in long tables.

Moreover, most instances from 
the \textit{Harwell-Boeing} and \textit{grid} datasets, as well as a big 
part of those from the \textit{Rome graphs} dataset are too big to be 
solved by the MILP-based approaches or using our algorithm in
reasonable time.

\subsubsection{Numerical results on Erd\H{o}s-R{\'e}nyi graphs and random geometric graphs}
In this section we provide a detailed evaluation of the
numerical experiments on the graphs with varying density described in
Section~\ref{sec:densegraphs}.}

Figure~\ref{fig:boundevolvement} shows the evolution of the bound computed by
our algorithm over seven iterations considering different sets of constraints on the
example of an Erd\H{o}s-R{\'e}nyi graph on $n=20$ vertices.
Note that in the initial iteration no constraints are added and in the
first iteration always only the 3-dicycle equations~\eqref{eq:3dicycleSDP} are
considered. The circles in Figure~\ref{fig:boundevolvement} show that the
bound does not move anymore from 
iteration three on if only considering constraints~\eqref{eq:3dicycleSDP}.
Adding the triangle inequalities~\eqref{eq:triangle} to the pool of
possible constraints leads to a remarkable increase of the bound,
cf. the triangles in the plot. 
And offering additionally the inequalities obtained from lifting the
linear ordering polytope~\eqref{eq:lin3d} and the inequalities from the
squared linear ordering polytope of order four~\eqref{eq:LOtwofour} 
further pushes the bound above with neglectable extra cost, cf. the
stars in Figure~\ref{fig:boundevolvement}. 
The plots are similar for all the instances we consider. Hence, these
experiments confirm our choice of using all the
constraints~\eqref{eq:3dicycleSDP}, \eqref{eq:triangle},
\eqref{eq:LOtwofour} and~\eqref{eq:lin3d} as
potential strengthenings within our algorithm.

\begin{figure}[tbph]
	\caption{Evolution of the bound over the iterations of Algorithm~\ref{alg:CW}, \\ 
	Erd\H{o}s-R{\'e}nyi graph with $n=20$ and $p = 0.8$, $7 \text{ 
	iterations}$}\label{fig:boundevolvement}
	\centering
	\includegraphics[width=0.75\textwidth]{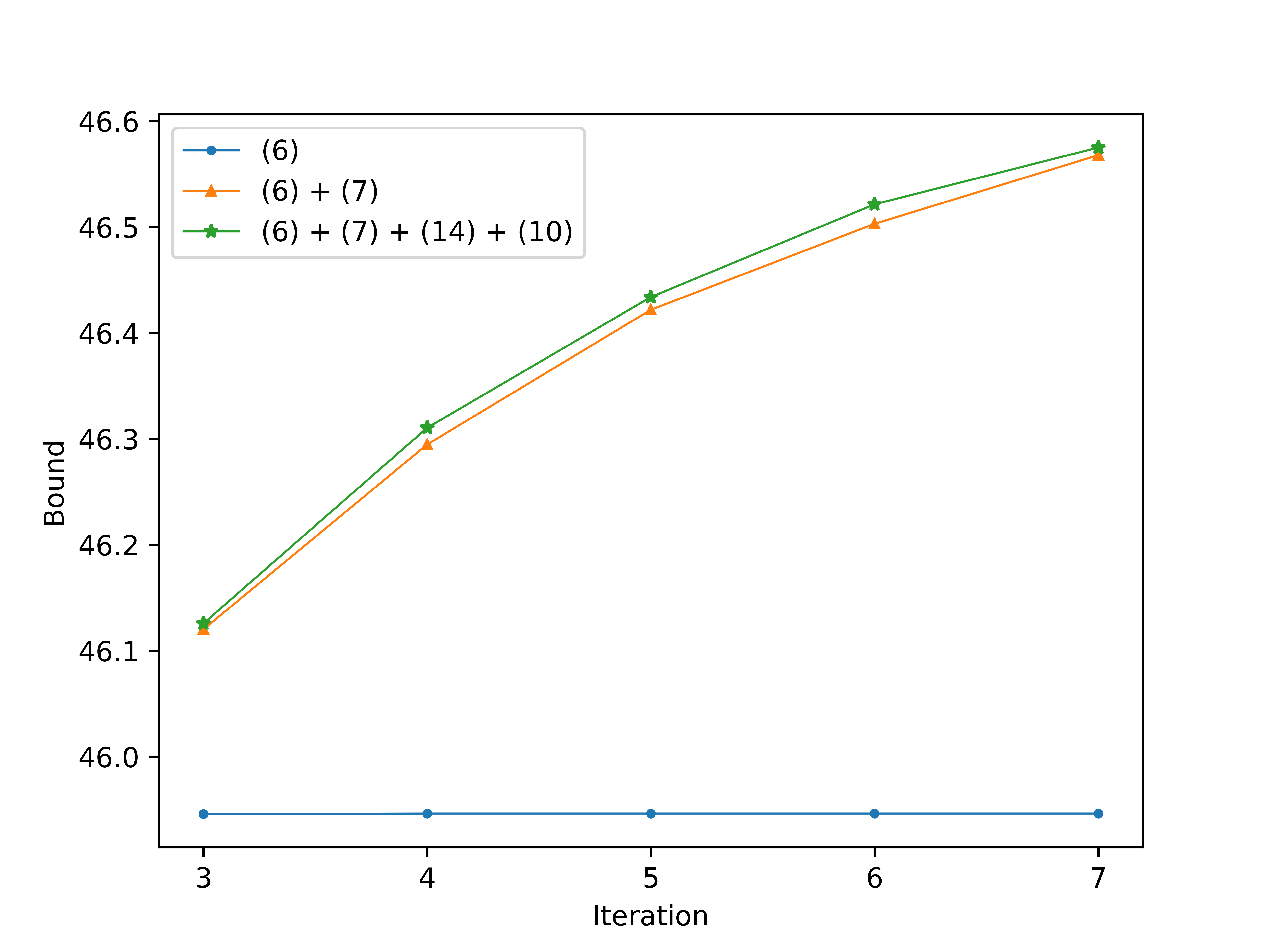}
\end{figure}

In Tables~\ref{tab:randomgraphs-sdp}--\ref{tab:geometricgraphs-ilp} we give
the numerical results for the above mentioned instances, comparing our
SDP bounds with those obtained when using an MILP solver. The column
label $n$ refers to the number of vertices of the graph, $p$ and $d$
relates to the edges of the graph as specified in
Section~\ref{sec:instances} above.

In Tables~\ref{tab:randomgraphs-sdp} and~\ref{tab:geometricgraphs-sdp}
columns ``LB init'' list the lower bounds
obtained when solving the initial basic SDP relaxation without any cutting
planes, while in column ``LB final'' the lower bound we finally
obtained is displayed. ``UB'' is the upper bound obtained through our
heuristic and ``gap final'' is computed as $(\mathrm{UB} -
\mathrm{LB})/\mathrm{UB}$, where LB is the final lower bound.
In the column with the ``time'' label, we report the total time for
obtaining our lower bound, and how this total time is split into
solving the SDPs and separating the violated inequality and equality 
constraints. Moreover,
the time for computing the upper bound is given.
The last column of these two tables indicate the number of cutting
planes added when the algorithm stops.
Tables~\ref{tab:randomgraphs-ilp} and~\ref{tab:geometricgraphs-ilp}
give the details when using CPLEX to obtain
an (approximate) solution. 

In the tables we see that adding the valid inequalities and equalities 
significantly
improves the initial bound, in Table~\ref{tab:randomgraphs-sdp} 
the value of the bound increases by roughly 30\%, for the random geometric graphs
(Table~\ref{tab:geometricgraphs-sdp}) it is 
around~20 to 30\%. Also the upper bounds we obtained are of good
quality, overall the gap between our final lower bounds and the upper
bounds is from 30\% to 46\% for the Erd\H{o}s-R{\'e}nyi graphs and
between 27\% and 53\%
for the random geometric graphs.

The time spent to compute the lower bounds ranges from a few seconds
for the instances with 20~vertices up to 115~minutes for instances
with 50~vertices. For the small instances, solving the SDPs can be done rather 
quickly and therefore, the time spent in
the separation takes the bigger ratio. For larger instances, however,
separating violated inequalities and equalities typically uses less than 20\% of the 
overall
time. Computing the upper bounds is done within less than 40~seconds.

We now turn our attention to comparing our results with using the MILP
solver CPLEX, see Tables~\ref{tab:randomgraphs-ilp}
and~\ref{tab:geometricgraphs-ilp}. As all our SDP based results are
obtained within 
2~hours, we set this as time limit for CPLEX. On small and sparse
instances CPLEX performs very well. However, for the
Erd\H{o}s-R{\'e}nyi graph with $n=40$ and 70\% density we are left with a
gap of 45\% after two hours,
compared to the gap of 34\% that we obtain with our algorithm in less than 30~minutes.
The random geometric graphs with $d$ equal to~$0.3$ or $0.4$ can be
solved for graphs with up to 40~vertices, for dense instances with at least 
$40$ vertices a gap of
more than 35\% remains after two hours run time. In general, the
density has a huge impact on the runtime for the MILP solver. For example, it
increases from 3.45~seconds for $p=0.3$ to 1206.53~seconds for $p=0.9$
for Erd\H{o}s-R{\'e}nyi graphs on 20~vertices. Our SDP based bounds do not show
significant differences for sparse and dense instances for the
Erd\H{o}s-R{\'e}nyi graphs, and only a minor increase of the runtime for
the random geometric graphs.

\revision{
	To get a better understanding of the quality of our lower and upper bounds, we 
	computed the optimal solutions for some instances by running the 
	MILP-model up to several days of CPU-time. The results are displayed in 
	Tables \ref{tab:erdosrenyi-exactvalues} and 
	\ref{tab:geometricgraphs-exactvalues}, in which the columns labeled $n$, $p$, 
	$d$ and ``density'' relate to the same instance parameters as previously 
	specified. Columns ``LB'' and ``UB'' list the lower and upper 
	bounds, respectively, obtained by our algorithm.
        In column ``opt'' the optimal value 
	of the CMP obtained with the MILP-model is given, and the last column 
	shows the deviation of our lower bound from the optimum, i.e., the value 
	$(\text{opt}- \text{LB}) / \text{opt}$.
	These tables show that, for most of these instances, our upper 
	bounds 
	are rather close to the optimal value, and consequently that most of 
	the gap comes from the lower bound. Indeed,  
	our lower bound is always approximately 30\% away from the optimum for 
	the Erd\H{o}s-R{\'e}nyi graphs, and most of the time
	between 20\% and 30\% for the geometric graphs. As this behavior seems 
	consistent, we can presume that it stays similar for the bigger and denser
	instances, for which we cannot compute the optimum in reasonable time. 
	Another observation is that for the smallest 
	instances the quality of the upper bound increases with 
	the density of the graph, but it is unclear whether this is also the 
	case for bigger instances.
}

Overall, the MILP approach is clearly outperformed by our method for
dense graphs but also for graphs having~40 or 50~vertices, independent
of their density.

\begin{table}[tbph]
		\caption{Bounds and times for the MCP using the SDP relaxation
		on Erdős–Rényi graphs}
	\label{tab:randomgraphs-sdp}
	\centering
	\small
	
	\begin{tabular}{@{}ccrrrrrrrrr@{}}
		\toprule
		&                      & 
		\multicolumn{4}{c}{bounds}                                            
		                                    &
		 
		\multicolumn{4}{c}{time}                                              
		                                  &
		 \multicolumn{1}{c}{}       \\ \cmidrule(lr){3-6} \cmidrule(lr){7-10}
		&                      & \multicolumn{1}{c}{LB}   & 
		\multicolumn{1}{c}{LB}    & \multicolumn{1}{c}{}   & 
		\multicolumn{1}{c}{gap}   & \multicolumn{1}{c}{LB}    & 
		\multicolumn{1}{c}{LB}  & \multicolumn{1}{c}{LB}   & 
		\multicolumn{1}{c}{}   & \multicolumn{1}{c}{}       \\
		$n$                   & $p$                    & 
		\multicolumn{1}{c}{init} & \multicolumn{1}{c}{final} & 
		\multicolumn{1}{c}{UB} & \multicolumn{1}{c}{final} & 
		\multicolumn{1}{c}{total} & \multicolumn{1}{c}{SDP} & 
		\multicolumn{1}{c}{sep.} & \multicolumn{1}{c}{UB} & 
		\multicolumn{1}{c}{\#cuts} \\ \midrule
		20                   & 0.3                  & 
		10.99                    & 14.27                     & 
		26                     & 0.42                      & 
		62.62                     & 23.37                   & 
		38.28                    & 0.36                   & 
		3697                       \\
		20                   & 0.4                  & 
		17.05                    & 21.54                     & 
		36                     & 0.39                      & 
		61.83                     & 22.61                   & 
		38.28                    & 0.32                   & 
		3826                       \\
		20                   & 0.5                  & 
		17.14                    & 21.92                     & 
		41                     & 0.46                      & 
		63.94                     & 23.47                   & 
		39.40                    & 0.44                   & 
		3919                       \\
		20                   & 0.6                  & 
		28.89                    & 36.40                     & 
		55                     & 0.33                      & 
		63.38                     & 23.23                   & 
		39.02                    & 0.51                   & 
		3832                       \\
		20                   & 0.7                  & 
		30.59                    & 38.65                     & 
		57                     & 0.32                      & 
		63.59                     & 23.37                   & 
		39.20                    & 0.39                   & 
		3806                       \\
		20                   & 0.8                  & 
		37.12                    & 46.59                     & 
		68                     & 0.31                      & 
		62.43                     & 21.74                   & 
		39.45                    & 0.62                   & 
		3789                       \\
		20                   & 0.9                  & 
		46.87                    & 59.47                     & 
		88                     & 0.32                      & 
		59.05                     & 19.58                   & 
		38.37                    & 0.49                   & 
		3267                       \\
		\multicolumn{1}{l}{} & \multicolumn{1}{l}{} & 
		\multicolumn{1}{l}{}     & \multicolumn{1}{l}{}      & 
		\multicolumn{1}{l}{}   & \multicolumn{1}{l}{}      & 
		\multicolumn{1}{l}{}      & \multicolumn{1}{l}{}    & 
		\multicolumn{1}{l}{}     & \multicolumn{1}{l}{}   & 
		\multicolumn{1}{l}{}       \\
		30                   & 0.3                  & 
		23.65                    & 30.69                     & 
		48                     & 0.35                      & 
		365.83                    & 206.53                  & 
		153.80                   & 2.21                   & 
		9364                       \\
		30                   & 0.4                  & 
		32.20                    & 41.79                     & 
		76                     & 0.45                      & 
		370.88                    & 212.26                  & 
		153.67                   & 1.67                   & 
		9244                       \\
		30                   & 0.5                  & 
		47.39                    & 60.55                     & 
		95                     & 0.36                      & 
		376.59                    & 215.95                  & 
		155.00                   & 2.33                   & 
		9687                       \\
		30                   & 0.6                  & 
		55.72                    & 70.84                     & 
		101                    & 0.30                      & 
		400.35                    & 235.57                  & 
		158.45                   & 3.01                   & 
		10149                      \\
		30                   & 0.7                  & 
		71.93                    & 91.86                     & 
		143                    & 0.36                      & 
		379.24                    & 220.62                  & 
		152.50                   & 2.65                   & 
		9160                       \\
		30                   & 0.8                  & 
		86.47                    & 110.58                    & 
		168                    & 0.34                      & 
		401.19                    & 237.39                  & 
		156.08                   & 4.33                   & 
		8949                       \\
		30                   & 0.9                  & 
		93.63                    & 119.93                    & 
		181                    & 0.34                      & 
		392.83                    & 233.53                  & 
		153.09                   & 2.90                   & 
		9106                       \\
		\multicolumn{1}{l}{} & \multicolumn{1}{l}{} & 
		\multicolumn{1}{l}{}     & \multicolumn{1}{l}{}      & 
		\multicolumn{1}{l}{}   & \multicolumn{1}{l}{}      & 
		\multicolumn{1}{l}{}      & \multicolumn{1}{l}{}    & 
		\multicolumn{1}{l}{}     & \multicolumn{1}{l}{}   & 
		\multicolumn{1}{l}{}       \\
		40                   & 0.3                  & 
		46.15                    & 60.12                     & 
		92                     & 0.34                      & 
		1505.80                   & 1092.75                 & 
		396.12                   & 4.97                   & 
		17300                      \\
		40                   & 0.4                  & 
		66.42                    & 85.52                     & 
		126                    & 0.32                      & 
		1597.97                   & 1169.59                 & 
		407.50                   & 8.91                   & 
		18712                      \\
		40                   & 0.5                  & 
		77.24                    & 99.75                     & 
		155                    & 0.35                      & 
		1495.82                   & 1091.20                 & 
		387.26                   & 5.41                   & 
		16655                      \\
		40                   & 0.6                  & 
		101.62                   & 130.84                    & 
		195                    & 0.33                      & 
		1593.94                   & 1176.42                 & 
		399.59                   & 5.87                   & 
		17849                      \\
		40                   & 0.7                  & 
		120.73                   & 155.30                    & 
		235                    & 0.34                      & 
		1691.67                   & 1271.61                 & 
		397.09                   & 10.94                  & 
		17663                      \\
		40                   & 0.8                  & 
		146.23                   & 187.95                    & 
		281                    & 0.33                      & 
		1723.79                   & 1299.59                 & 
		400.07                   & 12.10                  & 
		18293                      \\
		40                   & 0.9                  & 
		176.23                   & 227.74                    & 
		346                    & 0.34                      & 
		1703.14                   & 1291.81                 & 
		391.63                   & 7.64                   & 
		16651                      \\
		\multicolumn{1}{l}{} & \multicolumn{1}{l}{} & 
		\multicolumn{1}{l}{}     & \multicolumn{1}{l}{}      & 
		\multicolumn{1}{l}{}   & \multicolumn{1}{l}{}      & 
		\multicolumn{1}{l}{}      & \multicolumn{1}{l}{}    & 
		\multicolumn{1}{l}{}     & \multicolumn{1}{l}{}   & 
		\multicolumn{1}{l}{}       \\
		50                   & 0.3                  & 
		70.48                    & 91.92                     & 
		159                    & 0.42                      & 
		5310.93                   & 4359.63                 & 
		901.09                   & 16.17                  & 
		30416                      \\
		50                   & 0.4                  & 
		102.90                   & 133.61                    & 
		211                    & 0.36                      & 
		5404.70                   & 4436.08                 & 
		905.93                   & 28.63                  & 
		30238                      \\
		50                   & 0.5                  & 
		140.54                   & 182.44                    & 
		286                    & 0.36                      & 
		5313.71                   & 4396.23                 & 
		863.22                   & 20.16                  & 
		28586                      \\
		50                   & 0.6                  & 
		170.81                   & 221.14                    & 
		338                    & 0.34                      & 
		5511.41                   & 4594.26                 & 
		858.48                   & 24.56                  & 
		28417                      \\
		50                   & 0.7                  & 
		196.65                   & 253.98                    & 
		382                    & 0.34                      & 
		5552.34                   & 4610.50                 & 
		890.25                   & 17.42                  & 
		30954                      \\
		50                   & 0.8                  & 
		241.24                   & 312.20                    & 
		468                    & 0.33                      & 
		6303.39                   & 5365.25                 & 
		881.05                   & 22.80                  & 
		29710                      \\
		50                   & 0.9                  & 
		270.41                   & 350.87                    & 
		525                    & 0.33                      & 
		6033.16                   & 5106.08                 & 
		874.43                   & 18.43                  & 
		29418                      \\ \bottomrule
	\end{tabular}
\end{table}

\begin{table}[tbph]
		\caption{Bounds and times for the MCP using the MILP formulation
		on Erdős–Rényi graphs}
	\label{tab:randomgraphs-ilp}
	\centering
	\small
	\begin{tabular}{@{}ccrrrr@{}}
		\toprule
		$n$ & $p$ & \multicolumn{1}{c}{LB} & \multicolumn{1}{c}{UB} & \multicolumn{1}{c}{gap} & \multicolumn{1}{c}{time} \\ \midrule
		20  & 0.3 & 20.00                  & 20                     & 0.00                    & 3.45                     \\
		20  & 0.4 & 32.00                  & 32                     & 0.00                    & 118.43                   \\
		20  & 0.5 & 31.00                  & 31                     & 0.00                    & 11.81                    \\
		20  & 0.6 & 52.00                  & 52                     & 0.00                    & 25.58                    \\
		20  & 0.7 & 56.00                  & 56                     & 0.00                    & 33.04                    \\
		20  & 0.8 & 68.00                  & 68                     & 0.00                    & 271.60                   \\
		20  & 0.9 & 88.00                  & 88                     & 0.00                    & 1206.53                  \\
		&     &                        &                        &                         &                          \\
		30  & 0.3 & 44.00                  & 44                     & 0.00                    & 191.15                   \\
		30  & 0.4 & 60.00                  & 60                     & 0.00                    & 203.37                   \\
		30  & 0.5 & 87.00                  & 87                     & 0.00                    & 1165.32                  \\
		30  & 0.6 & 101.00                 & 101                    & 0.00                    & 1326.97                  \\
		30  & 0.7 & 123.33                 & 136                    & 0.09                    & 7200.00                  \\
		30  & 0.8 & 122.53                 & 163                    & 0.25                    & 7200.00                  \\
		30  & 0.9 & 134.75                 & 176                    & 0.23                    & 7200.00                  \\
		&     &                        &                        &                         &                          \\
		40  & 0.3 & 88.00                  & 89                     & 0.01                    & 7200.00                  \\
		40  & 0.4 & 115.37              & 121                    & 0.05                    & 7200.00                  \\
		40  & 0.5 & 105.67                 & 150                    & 0.30                    & 7200.00                  \\
		40  & 0.6 & 117.00                 & 193                    & 0.39                    & 7200.00                  \\
		40  & 0.7 & 127.00                 & 232                    & 0.45                    & 7200.00                  \\
		40  & 0.8 & 143.80                 & 283                    & 0.49                    & 7200.00                  \\
		40  & 0.9 & 150.13                 & 343                    & 0.56                    & 7200.00                  \\
		&     &                        &                        &                         &                          \\
		50  & 0.3 & 91.44                  & 142                    & 0.36                    & 7200.00                  \\
		50  & 0.4 & 105.00                 & 208                    & 0.50                    & 7200.00                  \\
		50  & 0.5 & 106.09                 & 283                    & 0.63                    & 7200.00                  \\
		50  & 0.6 & 113.00                 & 337                    & 0.66                    & 7200.00                  \\
		50  & 0.7 & 123.93                 & 388                    & 0.68                    & 7200.00                  \\
		50  & 0.8 & 123.00                 & 477                    & 0.74                    & 7200.00                  \\
		50  & 0.9 & 128.42                 & 536                    & 0.76                    & 7200.00                  \\
		\bottomrule
	\end{tabular}
\end{table}

\begin{table}[tbph]
		\caption{Bounds and times for the MCP using the SDP relaxation
		on random geometric graphs}
	\label{tab:geometricgraphs-sdp}
	\centering
	\small
	
	\begin{tabular}{@{}cccrrrrrrrrr@{}}
		\toprule
		&                      &                      & 
		\multicolumn{4}{c}{bounds}                                            
		                                    & 
		\multicolumn{4}{c}{time}                                              
		                                  &
		 \multicolumn{1}{c}{}       \\ \cmidrule(lr){4-7} \cmidrule(lr){8-11}
		&                      &                      & 
		\multicolumn{1}{c}{LB}   & \multicolumn{1}{c}{LB}    & 
		\multicolumn{1}{c}{}   & \multicolumn{1}{c}{gap}   & 
		\multicolumn{1}{c}{LB}    & \multicolumn{1}{c}{LB}  & 
		\multicolumn{1}{c}{LB}   & \multicolumn{1}{c}{}   & 
		\multicolumn{1}{c}{}       \\
		$n$                  & $d$                  & density              & 
		\multicolumn{1}{c}{init} & \multicolumn{1}{c}{final} & 
		\multicolumn{1}{c}{UB} & \multicolumn{1}{c}{final} & 
		\multicolumn{1}{c}{total} & \multicolumn{1}{c}{SDP} & 
		\multicolumn{1}{c}{sep.} & \multicolumn{1}{c}{UB} & 
		\multicolumn{1}{c}{\#cuts} \\ \midrule
		20                   & 0.3                  & 0.18                 & 
		4.66                     & 5.24                      & 
		9                      & 0.33                      & 
		42.55                     & 7.02                    & 
		34.74                    & 0.26                   & 
		2592                       \\
		20                   & 0.4                  & 0.28                 & 
		7.16                     & 8.28                      & 
		13                     & 0.31                      & 
		42.69                     & 8.68                    & 
		33.19                    & 0.28                   & 
		2439                       \\
		20                   & 0.5                  & 0.42                 & 
		13.03                    & 16.50                     & 
		24                     & 0.29                      & 
		61.18                     & 20.23                   & 
		40.03                    & 0.32                   & 
		3576                       \\
		20                   & 0.6                  & 0.44                 & 
		14.21                    & 16.56                     & 
		24                     & 0.29                      & 
		57.48                     & 16.69                   & 
		39.87                    & 0.33                   & 
		2989                       \\
		20                   & 0.7                  & 0.72                 & 
		33.23                    & 39.83                     & 
		67                     & 0.40                      & 
		58.37                     & 17.32                   & 
		39.80                    & 0.64                   & 
		3053                       \\
		20                   & 0.8                  & 0.91                 & 
		45.49                    & 56.75                     & 
		83                     & 0.31                      & 
		62.43                     & 19.63                   & 
		41.73                    & 0.43                   & 
		2955                       \\
		20                   & 0.9                  & 0.83                 & 
		39.93                    & 48.57                     & 
		69                     & 0.29                      & 
		59.91                     & 19.82                   & 
		39.07                    & 0.41                   & 
		2631                       \\
		\multicolumn{1}{l}{} & \multicolumn{1}{l}{} & \multicolumn{1}{l}{} & 
		\multicolumn{1}{l}{}     & \multicolumn{1}{l}{}      & 
		\multicolumn{1}{l}{}   & \multicolumn{1}{l}{}      & 
		\multicolumn{1}{l}{}      & \multicolumn{1}{l}{}    & 
		\multicolumn{1}{l}{}     & \multicolumn{1}{l}{}   & 
		\multicolumn{1}{l}{}       \\
		30                   & 0.3                  & 0.19                 & 
		9.17                     & 10.46                     & 
		19                     & 0.42                      & 
		293.96                    & 125.34                  & 
		163.40                   & 2.08                   & 
		6305                       \\
		30                   & 0.4                  & 0.35                 & 
		26.04                    & 32.72                     & 
		65                     & 0.49                      & 
		337.93                    & 167.22                  & 
		165.32                   & 2.14                   & 
		8699                       \\
		30                   & 0.5                  & 0.48                 & 
		42.01                    & 51.65                     & 
		79                     & 0.34                      & 
		346.90                    & 176.51                  & 
		165.47                   & 1.67                   & 
		8342                       \\
		30                   & 0.6                  & 0.75                 & 
		76.20                    & 92.20                     & 
		136                    & 0.32                      & 
		373.08                    & 198.05                  & 
		167.51                   & 4.28                   & 
		7597                       \\
		30                   & 0.7                  & 0.70                 & 
		67.36                    & 84.98                     & 
		121                    & 0.30                      & 
		395.81                    & 223.55                  & 
		165.97                   & 2.98                   & 
		9897                       \\
		30                   & 0.8                  & 0.94                 & 
		105.86                   & 134.97                    & 
		200                    & 0.32                      & 
		391.51                    & 219.12                  & 
		166.81                   & 2.27                   & 
		8296                       \\
		30                   & 0.9                  & 0.96                 & 
		109.07                   & 139.47                    & 
		207                    & 0.32                      & 
		381.37                    & 206.85                  & 
		168.38                   & 2.88                   & 
		7682                       \\
		\multicolumn{1}{l}{} & \multicolumn{1}{l}{} & \multicolumn{1}{l}{} & 
		\multicolumn{1}{l}{}     & \multicolumn{1}{l}{}      & 
		\multicolumn{1}{l}{}   & \multicolumn{1}{l}{}      & 
		\multicolumn{1}{l}{}      & \multicolumn{1}{l}{}    & 
		\multicolumn{1}{l}{}     & \multicolumn{1}{l}{}   & 
		\multicolumn{1}{l}{}       \\
		40                   & 0.3                  & 0.21                 & 
		18.42                    & 22.48                     & 
		34                     & 0.32                      & 
		1253.51                   & 770.35                  & 
		463.22                   & 8.24                   & 
		12957                      \\
		40                   & 0.4                  & 0.31                 & 
		30.87                    & 39.98                     & 
		61                     & 0.34                      & 
		1389.39                   & 895.44                  & 
		477.01                   & 5.13                   & 
		15569                      \\
		40                   & 0.5                  & 0.44                 & 
		55.25                    & 69.41                     & 
		149                    & 0.53                      & 
		1353.95                   & 875.37                  & 
		456.78                   & 9.73                   & 
		15999                      \\
		40                   & 0.6                  & 0.55                 & 
		80.26                    & 97.37                     & 
		135                    & 0.27                      & 
		1416.42                   & 928.89                  & 
		468.95                   & 6.78                   & 
		14286                      \\
		40                   & 0.7                  & 0.74                 & 
		131.38                   & 165.83                    & 
		234                    & 0.29                      & 
		1650.37                   & 1178.59                 & 
		453.00                   & 6.85                   & 
		17886                      \\
		40                   & 0.8                  & 0.84                 & 
		158.23                   & 201.33                    & 
		299                    & 0.32                      & 
		1669.65                   & 1197.72                 & 
		448.71                   & 11.22                  & 
		18253                      \\
		40                   & 0.9                  & 0.95                 & 
		190.22                   & 245.12                    & 
		363                    & 0.32                      & 
		1758.20                   & 1290.30                 & 
		448.63                   & 7.24                   & 
		17238                      \\
		\multicolumn{1}{l}{} & \multicolumn{1}{l}{} & \multicolumn{1}{l}{} & 
		\multicolumn{1}{l}{}     & \multicolumn{1}{l}{}      & 
		\multicolumn{1}{l}{}   & \multicolumn{1}{l}{}      & 
		\multicolumn{1}{l}{}      & \multicolumn{1}{l}{}    & 
		\multicolumn{1}{l}{}     & \multicolumn{1}{l}{}   & 
		\multicolumn{1}{l}{}       \\
		50                   & 0.3                  & 0.26                 & 
		40.58                    & 48.00                     & 
		88                     & 0.44                      & 
		3534.43                   & 2325.10                 & 
		1157.76                  & 18.29                  & 
		18417                      \\
		50                   & 0.4                  & 0.38                 & 
		76.75                    & 96.25                     & 
		139                    & 0.30                      & 
		4139.52                   & 2889.61                 & 
		1197.40                  & 18.88                  & 
		22951                      \\
		50                   & 0.5                  & 0.47                 & 
		98.82                    & 127.56                    & 
		177                    & 0.28                      & 
		4433.95                   & 3256.09                 & 
		1128.30                  & 15.79                  & 
		26260                      \\
		50                   & 0.6                  & 0.61                 & 
		149.46                   & 189.41                    & 
		315                    & 0.40                      & 
		4604.08                   & 3441.39                 & 
		1109.84                  & 19.08                  & 
		27503                      \\
		50                   & 0.7                  & 0.76                 & 
		214.97                   & 272.53                    & 
		397                    & 0.31                      & 
		5064.41                   & 3916.95                 & 
		1089.12                  & 24.48                  & 
		28401                      \\
		50                   & 0.8                  & 0.79                 & 
		224.27                   & 283.43                    & 
		427                    & 0.33                      & 
		5512.39                   & 4352.64                 & 
		1086.82                  & 38.97                  & 
		29912                      \\
		50                   & 0.9                  & 0.93                 & 
		286.08                   & 368.44                    & 
		550                    & 0.33                      & 
		6983.07                   & 5810.03                 & 
		1118.61                  & 20.22                  & 
		31414                      \\ \bottomrule
	\end{tabular}
\end{table}

\begin{table}[tbph]
		\caption{Bounds and times for the MCP using the MILP formulation on 
		random geometric graphs}
	\label{tab:geometricgraphs-ilp}
	\centering
	\small
	\begin{tabular}{@{}ccrrrr@{}}
		\toprule
		$n$ & $d$ & \multicolumn{1}{c}{LB} & \multicolumn{1}{c}{UB} & 
		\multicolumn{1}{c}{gap} & \multicolumn{1}{c}{time} \\ \midrule
		20  & 0.3 & 7.00                   & 7                      & 
		0.00                    & 1.51                     \\
		20  & 0.4 & 12.00                  & 12                     & 
		0.00                    & 26.88                 \\
		20  & 0.5 & 22.00                  & 22                     & 
		0.00                    & 6.65                     \\
		20  & 0.6 & 21.00                  & 21                     & 
		0.00                    & 3.86                     \\
		20  & 0.7 & 55.00                  & 55                     & 
		0.00                    & 33.99                 \\
		20  & 0.8 & 83.00                  & 83                     & 
		0.00                    & 284.89                   \\
		20  & 0.9 & 69.00                  & 69                     & 
		0.00                    & 31.36                    \\
		&     &                        &                        
		&                         &                          \\
		30  & 0.3 & 13.00                  & 13                     & 
		0.00                    & 23.97                    \\
		30  & 0.4 & 45.00                  & 45                     & 
		0.00                    & 509.40                   \\
		30  & 0.5 & 60.31                  & 72                     & 
		0.16                    & 7200.00                  \\
		30  & 0.6 & 126.00                 & 126                    & 
		0.00                    & 1236.75                  \\
		30  & 0.7 & 119.00                 & 119                    & 
		0.00                    & 2401.15                  \\
		30  & 0.8 & 107.80                 & 200                    & 
		0.46                  & 7200.00                  \\
		30  & 0.9 & 120.40                 & 207                    & 
		0.42                    & 7200.00                  \\
		&     &                        &                        
		&                         &                          \\
		40  & 0.3 & 31.00                  & 31                     & 
		0.00                    & 1130.25                  \\
		40  & 0.4 & 50.00                  & 50                     & 
		0.00                    & 847.54                   \\
		40  & 0.5 & 92.00                  & 92                     & 
		0.00                    & 1241.06                  \\
		40  & 0.6 & 119.31                 & 126                    & 
		0.05                    & 7200.00                  \\
		40  & 0.7 & 154.66                 & 238                    & 
		0.35                    & 7200.00                  \\
		40  & 0.8 & 166.00                 & 294                    & 
		0.44                    & 7200.00                  \\
		40  & 0.9 & 130.35                 & 366                    & 
		0.64                    & 7200.00                  \\
		&     &                        &                        
		&                         &                          \\
		50  & 0.3 & 57.64                  & 63                     & 
		0.09                    & 7200.00                  \\
		50  & 0.4 & 115.50                 & 131                    & 
		0.12                    & 7200.00                  \\
		50  & 0.5 & 146.40                 & 175                    & 
		0.16                    & 7200.00                  \\
		50  & 0.6 & 162.33                 & 266                    & 
		0.39                    & 7200.00                  \\
		50  & 0.7 & 146.67                 & 398                    & 
		0.63                    & 7200.00                  \\
		50  & 0.8 & 112.01                 & 416                    & 
		0.73                    & 7200.00                  \\
		50  & 0.9 & 115.90                 & 556                    & 
		0.79                    & 7200.00                  \\ \bottomrule
	\end{tabular}
\end{table}

\begin{table}[tbph]
\revision{
	\parbox{.45\linewidth}{	
		\centering
		\small
		\caption{\small{Comparison of the SDP bounds with the optimal values 
		for the MCP	on Erdős–Rényi graphs}}
		\label{tab:erdosrenyi-exactvalues}
		\begin{tabular}{@{}ccrrrc@{}}
			\toprule
			$n$ & $p$& \multicolumn{1}{c}{LB} & 
			\multicolumn{1}{c}{UB} & \multicolumn{1}{c}{opt} & $\frac{\text{opt}- \text{LB}}{\text{opt}}$ \\ 
			\midrule
			20  & 0.3 & 14.27  & 26  & 20    & 0.25 \\
			20  & 0.4 & 21.54  & 36  & 32    & 0.31 \\
			20  & 0.5 & 21.92  & 41  & 31    & 0.29 \\
			20  & 0.6 & 36.40  & 55  & 52    & 0.29 \\
			20  & 0.7 & 38.65  & 57  & 56    & 0.30 \\
			20  & 0.8 & 46.59  & 68  & 68    & 0.31 \\
			20  & 0.9 & 59.47  & 88  & 88    & 0.32 \\
			&&&&& \\
			30  & 0.3 & 30.69  & 48  & 44    & 0.30 \\
			30  & 0.4 & 41.79  & 76  & 60    & 0.30 \\
			30  & 0.5 & 60.55  & 95  & 87    & 0.30 \\
			30  & 0.6 & 70.84  & 101 & 101   & 0.30 \\
			30  & 0.7 & 91.86  & 143 & 135   & 0.32 \\
			&&&&& \\
			40  & 0.3 & 60.12  & 92  & 89    & 0.31 \\
			40  & 0.4 & 85.52  & 126 & 121   & 0.29 \\
			40  & 0.5 & 99.75  & 155 & 147   & 0.32 \\ \bottomrule
		\end{tabular}}
\quad
\parbox{.45\linewidth}{
	\centering
	\small
	\caption{\small{Comparison of the SDP bounds with the optimal 
	values for 
		the MCP on geometric graphs}}
		\label{tab:geometricgraphs-exactvalues}
		\begin{tabular}{@{}cccrrrc@{}}
			\toprule
			$n$ & $d$ & density & \multicolumn{1}{c}{LB} & 
			\multicolumn{1}{c}{UB} & \multicolumn{1}{c}{opt} &  $\frac{\text{opt}- \text{LB}}{\text{opt}}$\\ 
			\midrule
			20  & 0.3 & 0.18    & 5.24                   & 
			9                      & 7                         & 0.14 \\
			20  & 0.4 & 0.28    & 8.28                   & 
			13                     & 12                        & 0.25 \\
			20  & 0.5 & 0.42    & 16.50                  & 
			24                     & 22                        & 0.23 \\
			20  & 0.6 & 0.44    & 16.56                  & 
			24                     & 21                        & 0.19 \\
			20  & 0.7 & 0.72    & 39.83                  & 
			67                     & 55                        & 0.27 \\
			20  & 0.8 & 0.91    & 56.75                  & 
			83                     & 83                        & 0.31 \\
			20  & 0.9 & 0.83    & 48.57                  & 
			69                     & 69                        & 0.29 \\
			&     &         &                        &                        
			&                           &      \\
			30  & 0.3 & 0.19    & 10.46                  & 
			19                     & 13                        & 0.15 \\
			30  & 0.4 & 0.35    & 32.72                  & 
			65                     & 45                        & 0.27 \\
			30  & 0.5 & 0.48    & 51.65                  & 
			79                     & 72                        & 0.28 \\
			30  & 0.6 & 0.75    & 92.20                  & 
			136                    & 126                       & 0.26 \\
			30  & 0.7 & 0.70    & 84.98                  & 
			121                    & 119                       & 0.29 \\
			&     &         &                        &                        
			&                           &      \\
			40  & 0.3 & 0.21    & 22.48                  & 
			34                     & 31                        & 0.26 \\
			40  & 0.4 & 0.31    & 39.98                  & 
			61                     & 50                        & 0.20 \\
			40  & 0.5 & 0.44    & 69.41                  & 
			149                    & 92                        & 0.24 \\
			\bottomrule         
		\end{tabular}
	}
}
	\end{table}

\section{Conclusion and outlook}\label{sec:conclusion}
We presented an SDP based approach to compute lower and upper bounds
on the cutwidth. In order to obtain tight bounds, we
derive several classes of valid inequalities and equalities and a heuristic for
separating those inequalities and equalities and add them iteratively in a 
cutting-plane
fashion to the SDP relaxation. The solution obtained from the SDP
relaxation also serves to 
compute an upper bound on the cutwidth. Our experiments show that we
obtain high quality bounds in reasonable time. In particular, our
method is by far not as sensitive to varying densities as MILP
approaches. 

As the number of vertices gets larger, solving the SDPs becomes the
time consuming part and thus the bottleneck. This is due to the
increasing size of the matrix, but even more due to the huge number of
constraints. Therefore, in our future research we will investigate
using alternating direction methods of multipliers (ADMM) instead of
an interior-point method, as ADMM proved to be successful in solving SDPs
with many constraints, see for example~\cite{demeijer2023admm}. 
Also including our bounds in a branch-and-bound framework and thereby having an 
exact solution method is a promising future research direction.

In our experiments we observe that we were not able to push the gap
significantly 
below $0.30$. We want to further investigate this to find theoretical
evidence of this behavior.
And finally, we want to apply our approach to computing related graph
parameters, i.e., those that can also be formulated in terms of
orderings of the vertices, like the treewidth or the pathwidth of a graph.

\clearpage
\bibliographystyle{plainnat}
\bibliography{papersCutwidth}

\end{document}